\documentclass[12pt]{article}
\usepackage{amsmath,amssymb,comment}
\newtheorem{theorem}{Theorem}
\newtheorem{proposition}{Proposition}
\newtheorem{lemma}{Lemma}
\newtheorem{corollary}{Corollary}
\newtheorem{definition}{Definition}

\newtheorem{example}{Example}
\newtheorem{algorithm}{Algorithm}
\numberwithin{equation}{section}
\newenvironment{proof}{\noindent Proof:}{$\Box$}

\newcommand{\N}{{\mathbb N}}
\newcommand{\Q}{{\mathbb Q}}

\newcommand{\C}{{\mathbb C}}
\newcommand{\Z}{{\mathbb Z}}
\newcommand{\ord}{{\mathrm{ord}}}
\newcommand{\Ann}{{\mathrm{Ann}}}
\newcommand{\pp}{{\mathfrak{p}}}
\newcommand{\qq}{{\mathfrak{q}}}

\newcommand{\Dsc}{{\mathcal D}}
\newcommand{\DscYX}{{\mathcal D}_{Y\rightarrow X}}
\newcommand{\DYX}{D_{Y\rightarrow X}}

\newcommand{\Bsc}{{\mathcal B}}
\newcommand{\Csc}{{\mathcal C}}
\newcommand{\Fsc}{{\mathcal F}}

\newcommand{\Gsc}{{\mathcal G}}
\newcommand{\Osc}{{\mathcal O}}

\newcommand{\Lsc}{{\mathcal L}}
\newcommand{\Isc}{{\mathcal I}}
\newcommand{\Jsc}{{\mathcal J}}
\newcommand{\Hsc}{{\mathcal H}}

\newcommand{\Msc}{{\mathcal M}}

\newcommand{\V}{{\mathbb V}}

\newcommand{\Spec}{\mbox{{\rm Spec}\,}}

\newcommand{\gr}{\mbox{{\rm gr}}}
\newcommand{\mvec}{{\bf m}}

\newcommand{\dx}{{\partial_x}}
\newcommand{\dy}{{\partial_y}}
\newcommand{\dz}{{\partial_z}}
\newcommand{\dt}{{\partial_t}}
\title{Indicial polynomials and $b$-functions
of $D$-modules along arbitrary varieties and their computation}
\author{Toshinori Oaku
}

\begin{document}
\maketitle

\begin{abstract}      
  We define an indicial polynomial of a $D$-module along an arbitrary
  subvariety as a generalization of both the classical indicial polynomial
  for a single linear differential equation
  and the Bernstein-Sato polynomial of a variety
  defined by Budur-Musta\c t\~a-Saito.
  An indicial polynomial is also a generalization of the $b$-function of
  a $D$-module along a submanifold and can be used in the computation of
  the $D$-module theoretic inverse image by the embedding instead of
  the $b$-function. 
  We consider properties of indicial polynomials and
  relations with $b$-functions.
  An indicial polynomial may exist even if the $b$-function does not,
  and gives the set of the roots of the $b$-function if it exists. 
  Computation of an indicial polynomial is easier than the $b$-function
  and naturally includes the case with parameters.
\end{abstract}

\section{Introduction---Definitions and basic properties}
\label{section:introduction}


We first define an indicial polyomial of a $D$-module along a submanifold
as a generalization of the $b$-function.
This notion can be generalized to an arbitrary subvariety,
possibly with singularities and/or multiplicities,
by using the technique introduced
by Budur, Musta\c t\~a, Saito \cite{BMS}.
This enables us to compute an indicial polynomial and the $b$-function
of an algebraic $D$-module along an arbitrary algebraic set.

Our interest in indicial polynomials 
comes from the following observations: 
\begin{enumerate}
\item
Compared with $b$-functions, indicial polynomials can exist
for a broader class of $D$-modules which are not necessarily holonomic
(see Example \ref{ex:nonholonomic}). 
\item
An indicial polynomial of a $D$-module
should give some information on the solutions of the $D$-module
as in the classical case of a single equation (see, e.g., \cite{Tahara},
\cite{Mandai}) 
if the $D$-module has regular singularities along a submanifold $Y$,
at least in case of codimension one. 
\item
  An indicial polynomial can be used for a concrete expression of 
  the $D$-module theoretic inverse image by the embedding of a
  submanifold even if the $b$-function does not exist
  (see Theorems \ref{th:inv}, \ref{th:algebraic_inv}, Examples
  \ref{ex:Fuchs}, \ref{ex:nonholonomic}). 
\item
  The computation of an indicial polynomial is easier than that of the
  $b$-function and gives the roots of the $b$-function if it exists.
  This serves as a, at least heuristic, method for computing the $b$-function
  with an associated stratification for e.g., holonomic $D$-modules,
  for which the existence of the $b$-function is assured
  (see Example \ref{ex:BMS}). 
\item
  Many important examples of holonomic systems
  such as $A$-hypergeometric systems 
  have parameters.
  It would be preferable for some purposes to regard parameters as
  variables.
  In that case,  it should be  natural to consider indicial polynomials instead of   $b$-functions (see Example \ref{ex:AHG}). 
\end{enumerate}

In this introduction, we restrict ourselves to the case of submanifold
and study general properties of indicial polynomials. 
Let us first recall definition of the $b$-function  
of a section of a $D$-module along a submanifold. 
Let $X$ be an $n$-dimensional complex manifold 
and $Y$ a complex analytic submanifold of $X$. 
Let $\Osc_X$ be the sheaf of holomorphic functions on $X$ and 
$\Isc_Y$ be the defining ideal of $Y$, which is a sheaf of ideals of 
$\Osc_X$. 
We denote $\Dsc_X$ the sheaf of (rings of) differential operators on $X$ 
with holomorphic coefficients. 

The {\em $V$-filtration} $\{V^j_Y(\Dsc_X)\}_{j\in\Z}$ of $\Dsc_X$ along $Y$ 
is the filtration on the sheaf $\Dsc_X|_Y$, which is the sheaf-theoretic 
restriction of $\Dsc_X$ to $Y$,  defined by
\[
 V^j_Y(\Dsc_X) := \{ P \in \Dsc_X|_Y \mid P\Isc_Y^i \subset \Isc_Y^{i-j} 
\quad (\forall i \in\Z) \}
\]
with the convention $\Isc_Y^j = \Osc_X|_Y$ for $j \leq 0$. 
Let $\theta$ be a vector field on a neighborhood in $X$ 
of each point of $Y$ 
which induces the identity map on $\Isc_Y/\Isc_Y^2$. 
In local coordinates $x = (x_1,\dots,x_d,x_{d+1},\dots,x_n)$ such that 
$Y = \{x_1=\cdots=x_d=0\}$, we may take 
\[
\theta = x_1\frac{\partial}{\partial{x_1}}+\cdots
+ x_d\frac{\partial}{\partial{x_d}}.
\]
Note that $\theta$ is uniquely determined modulo a vector field
which belongs to $V_Y^{-1}(\Dsc_X)$.

\begin{definition}\label{def:b-function}\rm
Let $\Msc$ be a coherent left $\Dsc_X$-module defined on $X$. 
Let $u$ be a section of $\Msc$ defined on a neighborhood of $p \in Y$. 
The 
{\em $b$-function} of $u$ along $Y$ at $p$ 
is the monic polynomial $b(s)$, if any, 
in an indeterminate $s$ of the least degree
such that
\[
 b(\theta)u \in V^{-1}_Y(\Dsc_X)_{p}u 
\]
holds, that is, there exists a germ $P$ of $\in V^{-1}_Y(\Dsc_X)$ at $p$ 
such that 
\[
(b(\theta)+P)u = 0. 
\]
If there exists such $b(s)$ for any $u \in \Msc_p$, then $\Msc$
is called {\em specializable} along $Y$ at $p$. 
\end{definition}

Define an ideal $\Jsc_Y(u)$ of the sheaf of rings $\Osc_Y[s]$ on $Y$ by 
\[
\Jsc_Y(u) := \{f(s,x'') \in \Osc_Y[s] \mid f(\theta, x'')u
\in V_Y^{-1}(\Dsc_X)u \}, 
\]
where $x''$ is a local coordinate system of $Y$. 
Then the $b$-function at $p$ is the monic generator of the ideal
$\Jsc_Y(u)_p \cap \C[s]$. 
The $b$-function of a general  $D$-module was introduced
with an existence proof for holonomic $D$-modules
by Kashiwara and Kawai \cite{KK1980}.
See also Laurent \cite{Laurent}. 
An algorithm to compute the $b$-function of a $D$-module
defined by differential operators with polynomial coefficients 
along a linear submanifold was given in \cite{OakuDuke} and
\cite{OTalgDmod}.

\begin{example}\rm
  Let $X = \C$ and 
  consider the $\Dsc_X$-module
$\Msc = \Dsc_X/\Dsc_X(x\partial_x^2 + (a+bx)\partial_x)$ with $u$ the residue
class of $1\in \Dsc_X$ in $\Msc$.
Then the $b$-function of $u$ along $\{0\}$ is $s(s+a-1)$, where we regard
$a,b \in \C$ as constants. 
\end{example}

The $b$-function is sometimes called the indicial polynomial
in connection with the classical theory of linear ordinary differential equations as in the example above. 
Here we propose to define the indicial polynomial as a slight generalization
of the $b$-function. 

\begin{definition}\label{def:indicial}\rm
Let $\Msc$ be a coherent left $\Dsc_X$-module defined on $X$. 
Let $u$ be a section of $\Msc$ defined on a neighborhood of $p \in Y$. 
Let $x'' = (x_{d+1},\dots,x_n)$ be a local coordinate system of $Y$
around $p$. 
An {\em indicial polynomial} of $u$ along $Y$ at $p$
is an element $b(s,x'')$ of $(\Osc_Y)_p[s]$ with the smallest degree and monic 
in $s$  
such that $b(s,x'') \in \Jsc_Y(u)_p$, i.e., 
\[
 b(\theta,x'')u \in V^{-1}_Y(\Dsc_X)_{p}u. 
\]
Then there exists a germ $P$ of $\in V^{-1}_Y(\Dsc_X)$ at $p$ 
such that 
\[
(b(\theta,x'')+P)u = 0. 
\]
\end{definition}

\begin{example}\rm
  Let $X = \C^2$ and $(x,y)$ be its coordinates.
  Consider the $\Dsc_X$-module
  $\Msc = \Dsc_X/\Dsc_X(x\partial_x^2 + y\partial_x + \partial_y^2)$
  with $u$ the residue
class of $1\in \Dsc_X$ in $\Msc$.
Then $s(s+y-1) \in (\Osc_Y)_p[s]$ is the indicial polynomial of $u$ 
along $Y = \{0\}\times \C$ at $p = (0,y_0)$. 
There is no $b$-function of $u$ along $Y$ at $p$. 
\end{example}

If $d=1$, 
this definition coincides with that of Baouendi-Goulaouic \cite{BG} 
for a single partial differential equation $Pu=0$ which is called
Fuchsian along $Y$.
The equation in the above example is Fuchsian along $Y$. 
The (multi-valued) analytic solutions around $Y$
of a Fuchsian equation $Pu = 0$ 
are described in terms of the indicial polynomial
as was shown by Tahara \cite{Tahara} under some condition on
the indicial polynomial, and then by Mandai \cite{Mandai}
in full generality. 

If $d=1$, the condition that there exists an indicial polynomial
of $u$ along $Y$ is called `elliptic along $Y$'
by Laurent and Schapira \cite{LS}.
In case $d > 1$, their ellipticity condition is weaker than
the existence of an indicial polynomial. 
They proved that every cohomology group of the inverse image to $Y$
of a $D$-module which is elliptic along $Y$ is coherent over $\Dsc_Y$. 
Laurent and Monteiro-Fernandes \cite{LM} defined Fuchsian $D$-modules 
along $Y$, which are, roughly speaking, elliptic and regular along $Y$.

An indicial polynomial coincides with the $b$-function if $Y$
is a point. 
An indicial polynomial is not uniquely determined in general; 
moreover, it does not necessarily coincide with
the $b$-function {\em even if $\Msc$ is holonomic.}

\begin{example}\rm 
Consider a system of differential equations
\[
(x\partial_x -y)u = y^2 u = 0
\]
on $X = \C^2 \ni (x,y)$ and let $\Msc = \Dsc_X u$ be the corresponding
$\Dsc_X$-module.
Then $\Msc$ is holonomic. 
It is easy to see that  $s -y +a(y)y^2$ is an indicial polynomial of $u$
along $Y = \{0\} \times \C$ at $(0,0)$ 
with an arbitrary germ $a(y)$ of $\Osc_Y$ at $0$,
while the $b$-function of $u$ along $Y$ at $(0,0)$ is $s^2$.
In fact $(x\partial_x)^2u = 0$ holds. 
On the other hand, at $(0,y_0)$ with $y_0 \neq 0$, 
the indicial polynomial and the $b$-function are both $1$.  
\end{example}

However, we have the following uniqueness.

\begin{proposition}
  In the notation of Definition \ref{def:indicial},
  assume that an indicial polynomial $b(s,x'')$ of $u$ along $Y$ at $p$ exists.
  Then
\begin{enumerate}
\item 
  $b(s,p) \in \C[s]$ is uniquely determined.
\item
  If $f(s,x'') \in (\Osc_Y)_p[s]$ belongs to $\Jsc_Y(u)_p$, then
  $f(s,p)$ is a multiple of $b(s,p)$ in $\C[s]$. 
  In particular, one has
    \[
  \V(\Jsc_Y(u)|_{x''=p}) :=
  \{ s \in \C \mid f(s,p) = 0 \,(\forall f(s,x'') \in \Jsc_Y(u)_p) \}
  = \{s \in \C \mid b(s,p) = 0\}. 
  \]
\end{enumerate}
\end{proposition}


\begin{proof} 
1. An indicial polynomial $b(s,x'')$ can be written in the form
\[
b(s,x'') = s^m + b_1(x'')s^{m-1} + \cdots + b_m(x'')
\]
with $b_i(x'') \in (\Osc_Y)_p$.
By the definition, $b(s,x'')$ 
is an element of $J := \Jsc_Y(u)_p$ in this form with the smallest $m$. 

It suffices to show that if the degree in $s$ of an element $f(s,x'')$ of
$J$
is less than $m$ then $f(s,p) = 0$ holds.
Let us prove this claim by induction on $m - \deg f(s,p)$.
Let $a(x'')s^d$ be the leading term of $f(s,x'')$ with respect to $s$.
If $a(p) \neq 0$, then $f(s,x)/a(x'')$ is monic and belongs to $J$.
This is a contradiction since $d < m$. Thus we have $a(p) = 0$. 

First assume  $\deg f(s,p) = m-1$. Then $d = m-1$ holds and
the coefficient of $s^d$ in $f(s,p)$ is $a(p) = 0$.
This is a contradiction. 
Next assume $\deg f(s,p) = k \leq m-2$. 
Then
  \[
  g(s,x'') := a(x'')b(s,x'') - s^{m-d}f(s,x'')
  \]
  belongs to $J$ and is of degree less than $m$ in $s$. 
Moreover, 
  $\deg g(s,p) = k+m-d > k$ holds since $a(p) = 0$. 
By the induction hypothesis $g(s,p) = 0$ holds, which implies $f(s,p) = 0$.

2. Assume $f(s,x'')$ belongs to $\Jsc_Y(u)$.
Dividing $f(s,x'')$ by $b(s,x'')$ in $(\Osc_Y)_p[s]$, write
\[
  f(s,x'') = q(s,x'')b(s,x'') + r(s,x'')
\]
with the degree of $r(s,x'')$ in $s$ less than $m$.
Then $r(s,p)$ vanishes by the argument above. 
This completes the proof. 
\end{proof}

Here is a general relation between the $b$-function and an indicial polynomial.

\begin{proposition}
  Assume that the $b$-funtion $b(s)$ of $u$ along $Y$ at $p \in Y$ exists.
  Then
  \begin{enumerate}
\item    
  An indicial polynomial $\tilde b(s,x'')$ of $u$ along $Y$ at $p$ exists
  and $\tilde b(s,p)$ divides $b(s)$ in $\C[s]$.
\item
  The set of the zeros of $\tilde b(s,p)$ coincides with
  that of $b(s)$. 
\end{enumerate}
\end{proposition}

\begin{proof} 
  (1)
  Set $J := \Jsc_Y(u)_p$, which is an ideal of $(\Osc_Y)_p[s]$. 
  The existence of $\tilde b(s,x'')$ follows from
  $b(s) \in (\Osc_Y)_p[s] \cap J$.
  Dividing $b(s)$ by $\tilde b(s,x'')$ in $(\Osc_Y)_p[s]$, we obtain
  \[
  b(s) = q(s,x'')\tilde b(s,x'') + c(s,x'')
  \]
  with $q(s,x''),c(s,x'') \in (\Osc_Y)_p[s]$ such that
  the degree of $c(s,x'')$ in $s$ is less than that of $\tilde b(s,x'')$.
  By the proof of the preceding proposition, $c(s,p) = 0$ holds
  since $c(s,x'') \in J$.

  (2) 
  Let $s_1,\dots,s_d$ be the distinct zeros of $\tilde b(s,p)$. 
Let us show that $b(s)$ has no other zeros. 
Suppose this is not the case and
  let $s_{d+1},\dots, s_{d+k}$ be the distinct zeros of $b(s)$
  other than $s_1,\dots,s_d$, with multiplicities  $m_1,\dots ,m_k$.  
Set
\[
b_1(s) :=(s-s_{d+1})^{m_1}\cdots (s-s_{d+k})^{m_k}, \quad
b_0(s) := b(s)/b_1(s).
\]
Then $b_1(s)$ and $\tilde b(s,x)$ both belongs to the ideal
\[
 B := J : b_0(s) = \{f(s,x'') \in (\Osc_Y)_p[s] \mid f(s,x'')b_0(s) \in J \}
\]
of $(\Osc_Y)_p[s]$. 
Let
\[
r(x) := \mathrm{Res}(\tilde b(s,x),b_1(s))
\]
be the resultant with respect to $s$.
Then $r(x)$ belongs to $B$.
Since $\tilde b(s,x)$ and $b_1(s)$ are monic in $s$, 
\[
r(p) = \mathrm{Res}(\tilde b(s,p),b_1(s))
\]
is a nonzero constant since $\tilde b(s,p)$ and $b_1(s)$ have no common zeros.
This implies $B = (\Osc_Y)_p[s]$ and hence $b_0(s)$ belongs to $J$,
which is a contradiciton.
\end{proof}

\begin{corollary}
  If there exists an indicial polynomial of $u$ along $Y$ at $p$
  that belongs to $\C[s]$, then it is the $b$-function of $u$ at $p$. 
\end{corollary}

\section{Indicial polynomials and  $b$-functions with respect to coherent ideals} 
\label{section:variety}

We follow the method by Budur-Musta\c t\~a-Saito \cite{BMS}
of defining the Bernstein-Sato polynomial for an arbitrary subvariety, possibly with multiplicities, in a slightly more general settings. 

Let $X$ be an $n$-dimensional complex manifold and 
let $\Fsc$ be a coherent sheaf of ideals of $\Osc_X$. 
Let $p$ be a point of the support of $\Osc_X/\Fsc$. 
Let $f_1,\dots,f_d$ be a set of local generators of $\Fsc$ on 
a Stein open neighborhood  $U$ of $p$.
Then $f_i(p) = 0$ holds for $i=1,\dots,d$.

Consider the associated graph embedding 
\[
\iota: U \ni x \longmapsto (x,f_1(x),\dots,f_d(x)) \in U \times \C^d .
\]
Then 
\[
Z := \iota(U) = \{(x,t_1,\dots,t_d) \in U \times \C^d  \mid
t_i = f_i(x) \quad (i=1,\dots,d)\}
\]
is a (non-singular) closed submanifold of $\C^d \times U$, 
which depends on the choice of local generators $f_1,\dots,f_d$ of $\Fsc$.  
Let $\Isc_Z$ be the defining ideal of $Z$ and set
\[
\Gamma_{[Z]}(\Gsc) := \{ u \in \Gsc \mid \Isc_Z^ku = 0
\quad(\exists k \in \N)\}
\]
for a $\Osc_{U\times\C^d}$-module $\Gsc$. 
Let $\Bsc_{Z|U \times \C^d} := \Hsc^d_{[Z]}(\Osc_{U\times \C^d})$ be  
the $d$-th algebraic local (or relative) cohomology group
of $\Osc_{U\times\C^d}$ supported by $Z$,
where $\Hsc^d_{[Z]}$ is the $d$-th right derived functor of
$\Gamma_{[Z]}$. 
Then $\Bsc_{Z|U \times \C^d}$  has a natural strucuture of 
holonomic left $\Dsc_{X\times\C^d}$-module (see Kashiwara \cite{KashiwaraII},
\cite{KashiwaraBook}). 
Suppose that $\Msc$ is a coherent left $\Dsc_X$-module defined on $U$.
Then the $D$-module theoretic direct image of $\Msc$ associated with
the embedding $\iota$ is the left $\Dsc_{U\times \C^d}$-module  
\[
 \iota_*(\Msc) := \Msc\otimes_{\Osc_X|_U}\Bsc_{Z|U\times \C^d}. 
 \]
 The functor $\iota_*$ is exact and sends a coherent (resp.\ holonomic)
 $\Dsc_U$-module to a coherent (resp.\ holonomic) $\Dsc_{U\times \C^d}$-module.

Let $u$ be a section of $\Msc$ defined on a neighborhood of $p$. 
Then   
\[
\iota_*(u) := u\otimes\delta(t_1-f_1)\cdots\delta(t_d-f_d)
\]
is defined as a section of $\iota_*(\Msc)$.
Here $\delta(t_1-f_1)\cdots\delta(t_d-f_d)$ is the residue
class
\[
\left[\frac{1}{(t_1-f_1)\cdots (t_d-f_d)} \right]
\in 
 \Osc_{U\times \C^d}[h^{-1}]
  /\sum_{j=1}^d \Osc_{U\times \C^d}[h_j^{-1}]
\simeq \Hsc^d_{[Z]}(\Osc_{U\times\C^d})
  \]
with $h := (t_1-f_1)\cdots(t_d-f_d)$ and
$h_j := h/(t_j-f_j)$. 
If $\Msc$ is generated by $u$, then $\iota_*(\Msc)$ is generated by
$\iota_*(u)$. 
See Algorithm \ref{alg:embedding} for a concrete description of
$\iota_*(\Msc)$ in this case. 
The first statement of the following theorem is due to
Budur-Musta\c t\~a-Saito \cite{BMS}. 

\begin{theorem}\label{th:BMS}
In the notation above
\begin{enumerate}
\item    
Assume that there exists the $b$-function $b(s)$ of $\iota_*(u)$ along
$U \times \{0\}$ at $\iota(p)$. 
Then $b(s-d)$ does not depend on the choice of local generators 
$f_1,\dots,f_d$ of $\Fsc$. 
We call $b(s-d)$ the {\em $b$-function of $u$} at $p$  
with respect to  $\Fsc$.  
\item    
The set
\begin{multline*}
B_p(u,\Fsc) := \{b(s,x) \in (\Osc_X)_p[s] \mid
\\
\mbox{$b(s+d,x)$ is an indicial polynomial of $\iota_*(u)$ along
  $U \times \{0\}$ at $\iota(p)$} \}
\end{multline*}
depends only on $u$ and $\Fsc$.
We call an element of $B_p(u,\Fsc)$ 
an {\em indicial polynomial of $u$} at $p$  
with respect to  $\Fsc$.  

\item
  If $b(s,x) \in B_p(u,\Fsc)$ and the degree of $c(s,x) \in \Fsc_p[s]$ in
  $s$ is less than that of $b(s,x)$, 
  then $b(s,x)+c(s,x)$ also belongs to $B_p(u,\Fsc)$. 
 
\end{enumerate}
\end{theorem}

\begin{proof}
Suppose that $\{f_1,\dots,f_d,f_{d+1}\}$ is also a set of local generators of 
$\Fsc$ around $p$. 
Then there exist sections $a_1,\dots,a_d$ of $\Osc_X$ defined on 
a Stein open neighborhood $V \subset U$ of $p$ such that
$f_{d+1} = a_1f_1 + \cdots + a_df_d$.
We may assume $V = U$ for the sake of the simplicity of the notation. 
Define an embedding
\[
\iota'\,:\, U \times \C^d \longrightarrow U \times \C^{d+1}
\]
by 
\[
\iota'(x,t_1,\dots,t_d) = (x,t_1,\dots,t_d,a_1(x)t_1 +\cdots+ a_d(x)t_d). 
\] 
Recall the definitions of $\iota : U \rightarrow U \times \C^d$ and 
\[
Z := \iota(U) = \{(x,t_1,\dots,t_d) \in V \times \C^d 
\mid t_i = f_i(x)\,(i=1,\dots,d)\}. 
\]
Define the embedding
\[
\tilde\iota \,:\, V \ni x \longmapsto
(x,f_1(x),\dots,f_d(x),f_{d+1}(x)) \in V \times \C^{d+1}. 
\]
Then $\tilde\iota = \iota'\circ\iota$ holds. This implies
$\tilde \iota_*(u) = \iota'_*(\iota_*(u))$. 
Thus the assertions 1,2 follow from the next lemma.
The assertion 3 follows from $(t_i - f_i(x))\iota_*(u) = 0$
for $i=1,\dots,d$. 
\end{proof}

\begin{lemma}\label{lemma:BMS}
  Let $\Msc$ be a coherent left $\Dsc_X$-module and $u$ a section of $\Msc$
  defined on a complex manifold $X$. 
Let $Y$ be a non-singular complex hypersurface of $X$ and let 
$\iota' : X \rightarrow X \times \C$ be a holomorphic embedding
defined by $\iota'(x) = (x,f(x))$ for $x \in X$ with a holomorphic
function $f$ on $X$.
Let $p$ be a point of $Y$ such that $f(p) = 0$.
Let $x''$ be a local coordinate system of $Y$ around $p$.
We regard $x''$ also as a local coordinate system of $\iota'(Y) \simeq Y$
around $\iota'(p) = (p,0)$. 
\begin{enumerate}
\item
Let $b(s)$ be the $b$-function of $u$ along $Y$ at $p$ and 
$b'(s)$ be that of $\iota'_*(u)$ along $\iota'(Y)$ at $(p,0)$. 
Then one has $b'(s-1) = b(s)$. 
In particular $b(s)$ exists if and only if $b'(s)$ exists. 
\item
  Let $b(s,x'') \in (\Osc_Y)_p[s]$
  be an indicial polynomial
  of $u$ along $Y$ at $p$.
  Then $b(s+1,x'')$ is an indicial polynomial of 
  $\iota'_*(u)$ along $\iota'(Y)$ at $(p,0)$.
Conversely, if $b'(s,x'') \in (\Osc_{\iota'(Y)})_p[s]$ is an indicial polynomial
of $\iota'_*(u)$ along $\iota'(Y)$ at $(p,0)$, 
then $b'(s-1,x'')$ is an
indicial polynomial of $u$ along $Y$ at $p$. 
\end{enumerate}
\end{lemma}

\begin{proof}
  We may assume $X$ is an open subset of $\C^n$ with $p=0\in \C^n$, 
\[
Y = \{x = (x_1,\dots,x_d,x_{d+1},\dots,x_n) \in X \mid
x_1 = \cdots = x_d = 0 \}, 
\] 
  and $\iota'(x) = (x,0)$.
  Then we have $\iota'_*(u) = u\otimes \delta(t)$ and 
\[
\iota'(Y) = \{(x,t) \in X\times\C \mid x_1=\cdots = x_d = t = 0\}
= Y \times \{0\} 
\]
with $x'' = (x_{d+1},\dots,x_n)$ a system coordinate of $Y$. 
We use the notation  $\partial = (\partial_1,\dots,\partial_n)$,
$\partial_i = \partial/\partial x_i$, $\partial_t = \partial/\partial t$.

Let us prove assertion 2. 
There exists $Q \in V_Y^{-1}(\Dsc_X)_p$ such that
\[
(b(x_1\partial_1+\cdots+ x_d\partial_d,x'') + Q)u = 0. 
\]
Then we have
\[
(b(x_1\partial_1+\cdots+ x_d\partial_d+\partial_tt,x'') + Q)\iota'_*(u)
= 0
\]
and $Q$ belongs to $V_{\iota'(Y)}^{-1}(\Dsc_{X\times\C})_{(p,0)}$. 
This implies that $b(s+1,x'')$ is an indicial polynomial of $\iota'_*(u)$
if it is of the smallest degree in $s$.

On the other hand, there exists $Q\in V_{\iota'(Y)}^{-1}(\Dsc_{X\times\C})_{(p,0)}$ 
such that
\[
(b'(x_1\partial_1 + \cdots+ x_d\partial_d + t\partial_t,x'') + Q)
\iota'_*(u) = 0. 
\]
Writing $Q$ in the form
\[
Q = \sum_{i,j\geq 0} Q_{ij}(x,\partial)\partial_t^it^j, 
\]
we have
\begin{align*}
  0 &= (b'(x_1\partial_1 + \cdots+ x_d\partial_d + t\partial_t,x'') + Q)
  \iota'_*(u) 
\\&
= b'(x_1\partial_1 + \cdots+ x_d\partial_d + \partial_tt-1,x'')
(u\otimes\delta(t))
 + \sum_{i,j\geq 0} Q_{ij}(x,\partial)u\otimes \partial_t^it^j\delta(t) 
\\&
= b'(x_1\partial_1 + \cdots+ x_d\partial_d -1,x'')u\otimes\delta(t)
+ \sum_{i\geq 0} Q_{i0}(x,\partial)u\otimes \delta^{(i)}(t) .
\end{align*}
This implies, in particular, 
\[
(b'(x_1\partial_1 + \cdots+ x_d\partial_d -1,x'') + 
Q_{00})u = 0
\]
with $Q_{00} \in V_Y^{-1}(\Dsc_X)_p$.
Thus $b'(s-1,x'')$ is an indicial polynomial of $u$ if it is of the smallest
degree in $s$. 

In conclusion, we have proved the equality
\begin{multline}
\{b(s-1,x'') \in (\Osc_Y)_p[s] \mid
(b(x_1\partial_1 + \cdots + x_d\partial_d,x'')u 
\in V_Y^{-1}(\Dsc_X)_p)u \}
\\=
\{b'(s,x'') \in (\Osc_Y)_p[s] \mid
(b'(x_1\partial_1 + \cdots + x_d\partial_d + t\partial_t,x'')+Q)
\iota'_*(u)
\in V_{\iota'(Y)}^{-1}(\Dsc_{X\times\C})_{(p,0)}\iota'_*(u) \}.
\nonumber
\end{multline}
The assertion 2 follows from this equality.
\end{proof}

If $\Msc = \Osc_X$ and $u=1$, then the $b$-function in the above sense 
conicides with $b(-s)$, where $b(s)$ is 
the Bernstein-Sato polynomial of the variety which corresponds to $\Fsc$  
defined by Budur-Musta\c t\~a-Saito \cite{BMS}. 

We shall show in the next section that
if $\Fsc$ is a defining ideal of a variety $Y$ and if $Y$ is non-singular
at $p$, then the $b$-function and an indicial polynomial at $p$
with respect to $\Fsc$ in Theorem \ref{th:BMS} correspond to those of
Definitions \ref{def:b-function} and \ref{def:indicial} along $Y$.

\section{Indicial polynomials and $b$-functions along submanifolds
via graph embedding}
\label{section:embedding}

Let $Y$ be a closed analytic subset of a complex manifold $X$. 
We denote $\Isc_Y$ the defining ideal of $Y$,
which is a coherent sheaf of ideals of $\Osc_X$. 
Let $u$ be a section of a coherent $\Dsc_X$-module $\Msc$
defined on a neighborhood of a non-singular point $p$ of $Y$.
Then at $p$, we can relate the $b$-function and an indicial polynomial of $u$
along $Y$ defined in Section 1
to those of $u$ with respect to $\Isc_Y$ defined in Section 2. 
This is crucial in actual computation. 

\begin{theorem}\label{th:embedding} 
In the notation above, let $d$ be the codimension of $Y$ around $p$.  
\begin{enumerate}
\item
The $b$-function $b(s)$ of $u$ along $Y$ at $p$ 
coincides with $b'(s+d)$ where $b'(s)$ is the $b$-function of $u$
with respect to $\Isc_Y$ at $p$. 
\item
  Let $b(s,x)$ be a monic polynomial in $(\Osc_X)_p[s]$ such that
  its restriction to $(\Osc_Y)_p[s]= (\Osc_X/\Isc_Y)[s]$,
  which we denote by $b(s,x)|_Y$,
  is an indicial polynomial of $u$ along $Y$.
  Then $b(s-d,x)$ is an indicial polynomial of $u$
  with respect to $\Isc_Y$ at $p$. 
  Conversely, if $b'(s,x)$ is an indicial polynomial of $u$
  with respect to $\Isc_Y$ at $p$, 
    then $b'(s+d,x)|_Y$ is an indicial polynomial of $u$
    along $Y$ at $p$.
\end{enumerate}
\end{theorem}

\begin{proof}
We may assume that $X$ is an open set of $\C^n$ with $p= 0 \in X$ 
and that there exists a set of local generators $f_1,\dots,f_d$ of
$\Isc_Y$ on an open subset $U$ of $X$ such that
$df_1\wedge\cdots\wedge df_d\neq 0$ at $p$. 
By a change of coordinates, we may assume that
$f_i = x_i$ for $i=1,\dots,d$.
We use the notation $x= (x',x'')$ with $x'=(x_1,\dots,x_d)$, 
$x''=(x_{d+1},\dots,x_n)$,   
and $\partial' = (\partial_1,\dots,\partial_d)$, 
$\partial'' = (\partial_{d+i},\dots, \partial_n)$ 
with $\partial_i = \partial/\partial x_i$. 
Then the embedding $\iota$ of the preceding section becomes 
\[
\iota\,:\, U \ni x= (x',x'') \longmapsto (x',x'',x') \in U \times \C^d.
\]

Let $b(s,x)$ be a monic polynomial in $(\Osc_X)_p[s]$
such that $b(s,x)|_Y = b(s,0,x'')$ is an indicial polynomial of $u$
along $Y$ at $p$.  
Then there exists a differential operator $Q = Q(x',x'',\partial',\partial'')
\in V_Y^{-1}(\Dsc_X)_0$ 
such that
\[
(b(x_1\partial_1 +\cdots+ x_d\partial_d,0,x'') + Q)u = 0.
\]
Expanding $Q$ in the form
\[
Q = \sum_{\alpha,\beta\in\N^d,|\beta|-|\alpha|\leq -1}
Q_{\alpha,\beta}(x'',\partial'')
x'^\alpha \partial'^\beta
\]
with $\N = \{0,1,2,\dots\}$, set
\[
\tilde Q :=
\sum_{\alpha,\beta\in\N^d,|\beta|-|\alpha|\leq -1}
Q_{\alpha,\beta}(x'',\partial'')
t^\alpha (\partial_t+\partial')^\beta
\]
with $t = (t_1,\dots,t_d)$ and
$\partial_t = (\partial_{t_1},\dots,\partial_{t_d})$. 
Then we have
\[
\tilde Q \iota_*(u) = \tilde Q(u\otimes\delta(t_1-x_1)\cdots\delta(t_d-x_d))
= Qu\otimes\delta(t_1-x_1)\cdots\delta(t_d-x_d)
\]
since
\begin{align*}
&
(t_i-x_i)(u\otimes\delta(t_1-x_1)\cdots\delta(t_d-x_d)) = 0, 
\\&
(\partial_{t_i} + \partial_i)(u\otimes\delta(t_1-x_1)\cdots\delta(t_d-x_d)) 
= \partial_iu\otimes\delta(t_1-x_1)\cdots\delta(t_d-x_d) 
\end{align*}
hold for $i=1,\dots,d$.
Hence we obtain
\begin{align*}
&
(b(t_1(\partial_{t_1}+\partial_1) +\cdots+ t_d(\partial_{t_d}+\partial_d),0,x'')
+ \tilde Q)\iota_*(u)
\\&
= (b(x_1\partial_1 + \cdots + x_d\partial_d,0,x'') + Q)u)
\otimes \delta(t_1-x_1)\cdots\delta(t_d-x_d)
= 0
\end{align*}
Here 
$\tilde Q$ 
and 
\[
R := b(t_1(\partial_{t_1}+\partial_1) +\cdots+ t_d(\partial_{t_d}+\partial_d),0,x'')
- b(t_1\partial_{t_1} +\cdots+ t_d\partial_{t_d},0,x'')
\] 
both belong to $V_{X\times\{0\}}^{-1}(\Dsc_{X\times\C^d})$. 
Thus it follows that 
$b(s-d,x)$ is an indicial polynomial of $u$ with respect
to $\Isc_Y$ at $p$ if it is of the smallest degree in $s$
in view of the assertion 3 of Theorem 1. 


Conversely,
assume that $b(s-d,x)$ is an indicial polynomial of $u$ with respect to
$\Isc_Y$ at $p$.
Then $b(s,x)$ is an indicial polynomial of  
$\iota_*(u)$ along $X\times\{0\}$ at $(0,0)$ by the definition. 
Hence there eixsts a differential operator
\[ 
Q = Q(t_1,\dots,t_d,x',x'',\partial_{t_1},\dots,\partial_{t_d},
\partial',\partial'') 
\in V_{X\times\{0\}}^{-1}(\Dsc_{X\times\C^d}) 
\]
such that
\[
 (b(t_1\partial_{t_1}+\cdots+t_d\partial_{t_d},x)+ Q)
(u\otimes\delta(t_1-x_1)\cdots\delta(t_d-x_d)) = 0. 
\]
By using the expansion
\[
Q = \sum_{\alpha,\beta\in\N^d,|\beta|-|\alpha|\leq -1}
Q_{\alpha,\beta}(x,\partial) t^\alpha\partial_t^\beta , 
\]
rewrite $Q$ in the form
\begin{align*}
Q &= 
  \sum_{\alpha,\beta\in\N^d,|\beta|-|\alpha|\leq -1}
  Q_{\alpha,\beta}(x,\partial) t^\alpha
  (\partial_t + \partial' - \partial')^\beta 
\\ &= 
  \sum_{\alpha,\beta\in\N^d,|\beta|-|\alpha|\leq -1}
Q'_{\alpha,\beta}(x,\partial) t^\alpha(\partial_t + \partial')^\beta . 
\end{align*}
Likewise writing $b(s,x)$ in the form
\[
b(s,x) = a_m(x)s^m + a_{m-1}(x)s^{m-1} + \cdots + a_0(x),\quad a_m(x) = 1,
\]
rewrite
\begin{align*}
b(t_1\partial_{t_1}+\cdots+ t_d\partial_{t_d},x)
&= \sum_{i=0}^m a_i(x)(t_1(\partial_{t_1}+\partial_1-\partial_1)+\cdots+
  t_d(\partial_{t_d}+\partial_d-\partial_d))^i 
\\&
= \sum_{i+j\leq m} a_{ij}(x)
(t_1\partial_{x_1}+\cdots + t_d\partial_{x_d})^j
(t_1(\partial_{t_1}+\partial_1)+\cdots + t_d(\partial_{t_d}+\partial_d))^i 
\\&
= b(t_1(\partial_{t_1}+\partial_1)+\cdots + t_d(\partial_{t_d}+\partial_d),x)
+  B
\end{align*}
with
\begin{align*}
&
  b(t_1(\partial_{t_1}+\partial_1)+\cdots + t_d(\partial_{t_d}+\partial_d),x)
  := \sum_{i=0}^m a_i(x)(t_1(\partial_{t_1}+\partial_1)+\cdots+
  t_d(\partial_{t_d}+\partial_d))^i,
\\&
B:= \sum_{i+j\leq m, j\geq 1} a_{ij}(x)
(t_1\partial_{x_1}+\cdots + t_d\partial_{x_d})^j
(t_1(\partial_{t_1}+\partial_1)+\cdots + t_d(\partial_{t_d}+\partial_d))^i 
\end{align*}
noting the commutation relation
\[
  [t_i(\partial_{t_i} + \partial_{x_i}),\,t_i\partial_{x_i}]
  = [t_i\partial_{t_i},\,t_i\partial_{x_i}]
  = t_i\partial_{x_i}.
 \]

 Now $P := B + Q$ belongs to $V_{X\times\{0\}}^{-1}(\Dsc_{X\times\C^d})$
 and has the form
\begin{equation}\label{eq:P}
 P =
   \sum_{\alpha,\beta\in\N^d,|\beta|-|\alpha|\leq -1}
P_{\alpha,\beta}(x,\partial) t^\alpha (\partial_t + \partial')^\beta
\end{equation}
and 
\begin{equation}\label{eq:bt}
(b(t_1(\partial_{t_1}+\partial_1)+\cdots + t_d(\partial_{t_d}+\partial_d),x)
+ P)\iota_*(u) = 0
\end{equation}
holds.
Let us denote by $\ord'$ the order of a differential  operator with respect to
$x'$, i.e., the total degree in $\partial'$.

Let us show that there exists $P \in V_{X\times\{0\}}^{-1}(\Dsc_{X\times\C^d})$
in the form (\ref{eq:P}) with $\ord' P_{\alpha,\beta} = 0$ such that 
the equation (\ref{eq:bt}) holds. 
To prove this claim, first note that $t_i-x_i$ commute with
$\partial_{t_i} + \partial_{x_i}$. Hence we have
\[
[t_i - x_i,\,P]\iota_*(u)
  = [t_i - x_i,\,b(t_1(\partial_{t_1}+\partial_1)+\cdots + t_d(\partial_{t_d}+\partial_d),x) + P]\iota_*(u) 
= 0.  
\]
Writing $P$ in the form (\ref{eq:P}) we get
\[
[t_i-x_i,\,P] = 
   \sum_{\alpha,\beta\in\N^d,|\beta|-|\alpha|\leq -1}
       [t_i-x_i,\,P_{\alpha,\beta}(x,\partial)]
       t^\alpha(\partial_t + \partial')^\beta  , 
\]
which belongs to $V_{X\times\{0\}}^{-1}(\Dsc_{X\times\C^d})$
and has a same form as (\ref{eq:P}).
Suppose that the largest degree of all the $P_{\alpha,\beta}(x,\partial)$
in $\partial_i$ is $k \geq 1$ for some $i \leq d$. 
Since $[t_i-x_i,\partial_{x_i}^k] =k\partial_{x_i}^{k-1}$,   
\[
P' := P - \frac1k [t_i-x_i,P]
\]
also belongs to $V_{X\times\{0\}}^{-1}(\Dsc_{X\times\C^d})$,
has a same form as in (\ref{eq:P}), and (\ref{eq:bt}) holds
with $P$ replaced by $P'$.
Moreover we have $\ord' P' < \ord' P$.
Thus we have proved the claim by induction.

Now we have the equation (\ref{eq:bt}) with some 
$P \in V_{X\times\{0\}}^{-1}(\Dsc_{X\times\C^d})$ of the form
\[
P =  \sum_{\alpha,\beta\in\N^d,|\beta|-|\alpha|\leq -1}
P_{\alpha,\beta}(x,\partial'') t^\alpha(\partial_t + \partial')^\beta
  .
\]
Then we have
\begin{multline*}
((b(x_1\partial_1+\cdots+x_d\partial_d,x) + \tilde P)u)
  \otimes\delta(t_1-x_1)\cdots\delta(t_d-x_d)
\\
= 
(b(t_1(\partial_{t_1} + \partial_1)+\cdots+t_d(\partial_{t_d}+\partial_d),x)
  +P)\iota_*(u)
=0
\end{multline*}
with
\begin{align*}
&
  b(x_1\partial_1 +\cdots + x_d\partial_d,x)
  := \sum_{i=0}^m a_i(x)(x_1\partial_1 + \cdots + x_d\partial_d)^i,
\\&
\tilde P :=  \sum_{\alpha,\beta\in\N^d,|\beta|-|\alpha|\leq -1}
P_{\alpha,\beta}(x,\partial'') x'^\alpha \partial'^\beta.  
\end{align*}
It is easy to see that $\tilde P$ belongs to $V_Y^{-1}(\Dsc_X)$. 
Thus $b(s,0,x'')$ is an indicial polynomial of $u$ along $Y$ at $p$
if it is of the smallest degree.
This completes the proof.
\end{proof}

\section{Inverse image of a $D$-module}
\label{section:inv}

In this section, we assume that $X$ is a $(n+d)$-dimensional
complex manifold and $Y$ is a closed submanifold of $X$ of codimension $d$.
Around a given point $p$ of $Y$, we choose a local coordinate system
$(x,t) = (x_1,\dots,x_n,t_1,\dots,t_d)$ such that $Y$ is defined by
$t_1 = \dots = t_d = 0$. 

Set
$\Dsc_{Y\rightarrow X} := \Osc_Y\otimes_{\Osc_X}\Dsc_X$,
which is a $(\Dsc_Y,\Dsc_X)$-bimodule on $Y$.
In the local coordinate system $(x,t)$, we have
\[
\Dsc_{Y\rightarrow X} \simeq \Dsc_X/(t_1\Dsc_X + \cdots + t_d\Dsc_X).
\]
The inverse image of a coherent $\Dsc_X$-module $\Msc$ with respect
to the embedding $\iota : Y\rightarrow X$ 
is defined by
\[
\mathbb{L}\iota^*(\Msc) :=
\Dsc_{Y\rightarrow X}\stackrel{\mathbb{L}}{\otimes}_{\Dsc_X}\Msc
\]
in the derived category $D^{-}(\mathrm{Mod}(\Dsc_Y))$ 
of complexes of left $\Dsc_Y$-modules bounded above
(see Kashiwara \cite{KashiwaraBook}).
It is also called the ($D$-module theoretic) restriction of $\Msc$ to $Y$. 

It follows from a theorem of Laurent-Schapira \cite{LS} that
if there exists an indicial polynomial along $Y$ for
each section of $\Msc$, then the cohomology groups of $\mathbb{L}\iota^*(\Msc)$
are coherent $\Dsc_Y$-modules.
In the sequel let us give a concrete expression of the cohomology groups.  

For an integer vector $\mvec = (m_1,\dots,m_k) \in \Z^k$, 
we define the shifted $V$-filtration
\[
 V_Y^i[\mvec]((\Dsc_X)^k) := V_Y^{i-m_1}(\Dsc_X) \oplus \cdots \oplus
V_Y^{i-m_k}(\Dsc_X)
\]
on the free module $(\Dsc_X)^k$.
The associaed graded module is defined by
\[
\gr_Y[\mvec]((\Dsc_X)^k) := \bigoplus_{i\in\Z}\gr_Y^i[\mvec]((\Dsc_X)^k)
\]
with
\[
\gr_Y^i[\mvec]((\Dsc_X)^k) := V_Y^i[\mvec]((\Dsc_X)^k)/V_Y^{i-1}[\mvec]((\Dsc_X)^k).
\]
We omit $[\mvec]$ if $\mvec$ is the zero vector. 
There exists an exact sequence 
\begin{equation} \label{eq:resolution}
  \cdots \longrightarrow (\Dsc_X^{r_{d+1}})_p
  \stackrel{\psi_{d+1}}{\longrightarrow} 
 (\Dsc_X^{r_{d}})_p 
\stackrel{\psi_{d}}{\longrightarrow} 
\cdots 
\stackrel{\psi_2}{\longrightarrow} (\Dsc_X^{r_1})_p
\stackrel{\psi_1}{\longrightarrow} (\Dsc_X^{r_0})_p
\stackrel{\varphi}{\longrightarrow}\Msc_p \longrightarrow 0
\end{equation}
of  left $\Dsc_X$-modules together with vectors 
$\mvec_1 \in \Z^{r_1}$, $...$, $\mvec_l \in \Z^{r_l}$ such that 
\[
\psi_{j+1}(V_Y^k[\mvec_{j+1}](\Dsc_X^{r_{j+1}})_p)
\subset V_Y^k[\mvec_j](\Dsc_X^{r_j})_p
\] 
holds for $j\geq 0$ with $\mvec_0$ the zero vector and that 
(\ref{eq:resolution}) induces an exact sequence
\begin{equation}\label{eq:Vresolution}
\cdots\rightarrow
V_Y^k[\mvec_{d+1}](\Dsc_X^{r_{d+1}})_p
\stackrel{\psi_{d+1}}{\longrightarrow} 
\cdots
\stackrel{\psi_2}{\longrightarrow}
V_Y^k[\mvec_1](\Dsc_X^{r_1})_p
\stackrel{\psi_1}{\longrightarrow}
V_Y^k(\Dsc_X^{r_0})_p
\stackrel{\varphi}{\longrightarrow}
V_Y^k(\Msc_p) \rightarrow 0
\end{equation}
with $V_Y^k(\Msc_p) := \varphi(V_Y^k(\Dsc_X^{r_0})_p)$ 
for any integer $k$.
See \cite{GO} for a constructive proof. 

The stalk of the inverse image $\mathbb{L}\iota^*(\Msc)_p$
is isomorphic to the complex
\[
  \cdots \longrightarrow (\DscYX^{r_{d+1}})_p
  \stackrel{\psi_{d+1}}{\longrightarrow} 
 (\DscYX^{r_{d}})_p 
\stackrel{\psi_{d}}{\longrightarrow} 
\cdots 
\stackrel{\psi_2}{\longrightarrow} (\DscYX^{r_1})_p
\stackrel{\psi_1}{\longrightarrow} (\DscYX^{r_0})_p
\longrightarrow 0
\]
in the derived category $D^-(\mathrm{Mod}((\Dsc_Y)_p))$
but $(\DscYX^{r_{j}})_p$ are not finitely generated as
left $(\Dsc_Y)_p$-modules.

For a vector $\mvec\in \Z^r$, let
$\{V_Y^k[\mvec](\DscYX^r)\}_{k\in\Z}$ be the filtration induced by
$\{V_Y^k[\mvec](\Dsc_X^r)\}_{k\in\Z}$.  In the local coordinates $(t,x)$,
we have a local isomorphism
\[
\DscYX \simeq \bigoplus_{\alpha\in\N^d}\Dsc_Y\partial_t^\alpha
\]
of left $\Dsc_Y$-module. Then in terms of this isomorphism we have
\[
V_Y^k[\mvec](\DscYX^r) \simeq
\bigoplus_{i=1}^r \bigoplus_{\alpha\in\N^d,|\alpha|\leq k-m_i}
  \Dsc_Y\partial_t^\alpha. 
  \]

In general, let $\Lsc$ be a $\Z[t]$-module and $\Lsc_j$ ($j\in\Z$)
be additive subgroups of $\Lsc$ such that
\[
 t_i\Lsc_j \subset \Lsc_{j-1} \quad (i=1,\dots,d,\, j\in \Z).
 \]
For an integer  $k$, we define the (shifted) Koszul complex
associated with $\{\Lsc_j\}$ by 
\[
K^\bullet[k](\{\Lsc_j\},t)\,:\,
0 \rightarrow 
\Lsc_{k+d} \otimes_\Z \stackrel{0}{\wedge}\Z^d 
\stackrel{\delta}{\longrightarrow}
\Lsc_{k+d-1}\otimes_\Z \stackrel{1}{\wedge}\Z^d 
\stackrel{\delta}{\longrightarrow}
\cdots
\stackrel{\delta}{\longrightarrow}
\Lsc_k \otimes_\Z \stackrel{d}{\wedge}\Z^d 
\rightarrow 0,
\]
where $\delta$ is defined by 
\[
\delta(u\otimes e_{i_1}\wedge\cdots\wedge e_{i_j}) 
= \sum_{l=1}^d t_lu\otimes e_l\wedge e_{i_1}\wedge\cdots\wedge e_{i_j}
\]
for a subset $\{i_1,\dots,i_j\}$ of $\{1,\dots,d\}$ 
with the unit vectors $e_1,\dots,e_d$ of $\Z^d$.   
Here $\Lsc_{k+d-j}\otimes\stackrel{j}{\wedge}\Z^d$ 
lies at the $(j-d)$th place of the complex.
If $\Lsc_j = \Lsc$ for any $j$, we denote this complex simply
$K^\bullet(\Lsc,t)$. 

The Koszul complex $K^\bullet(\Dsc_X,t)$ is isomorphic to $\DscYX$. 
Consequently we have an isomorphism
\[
\mathbb{L}\iota^*(\Msc) \simeq K^\bullet(\Dsc_X,t)\otimes_{\Dsc_X}\Msc
\simeq K^\bullet(\Msc,t). 
\]
In particular, the $j$th cohomology group $H^j(\mathbb{L}\iota^*(\Msc))$ vanishes
if $j < -d$. 
The free resolution (\ref{eq:resolution}) induces a double complex
\[
\cdots\longrightarrow K^\bullet((\Dsc_X^{r_{d+1}})_p,t)
\longrightarrow K^\bullet((\Dsc_X^{r_d})_p,t)
\longrightarrow \cdots \longrightarrow
K^\bullet((\Dsc_X^{r_0})_p,t)\rightarrow 0, 
\]
the total complex of which is isomorphic to $K^\bullet(\Msc_p,t)$
in $D^-(\mathrm{Mod}((\Dsc_Y)_p))$.  
Likewise, for any integer $k$, the total complex of the double complex
\[
\cdots\rightarrow K^\bullet[k](\{V_Y^j[\mvec_{d+1}](\Dsc_X^{r_{d+1}})_p)\},t)
\rightarrow \cdots \rightarrow
K^\bullet[k](\{V_Y^j[\mvec_0](\Dsc_X^{r_{0}})_p\},t)\rightarrow 0
\]
is isomorphic to $K^\bullet(V_Y^k(\Msc)_p,t)$; 
the total complex of the double complex
\[
\cdots\rightarrow K^\bullet[k](\{\gr_Y^j[\mvec_{d+1}](\Dsc_X^{r_{d+1}})_p\},t)
\rightarrow \cdots \rightarrow
K^\bullet[k](\{\gr_Y^j[\mvec_0](\Dsc_X^{r_{0}})_p\},t)\rightarrow 0 
\]
is isomorphic to $K^\bullet(\gr_Y^k(\Msc_p),t)$
in $D^-(\mathrm{Mod}((\Dsc_Y)_p))$.  

Before stating the main result, let us prepare

\begin{lemma}\label{lemma:bsx}
  Let $u$ be an nonzero element of $\Msc_p$ and assume there exists
  a monic polynomial $b(s,x)\in (\Osc_Y)_p[s]$
  such that $b(\theta,x)u \in V_Y^{-1}(\Dsc_X)_p u$. 
  Then for any integer $k$ and any element $v$ of $V_Y^k(\Dsc_X)_p u$,
  there exists an integer $j$ such that
  $b(\theta + k,x)^jv \in V_Y^{k-1}(\Dsc_X)_pu$. 
\end{lemma}

\begin{proof}
  There exists $P \in V_Y^k(\Dsc_X)_p$ such that $v = Pu$.
  It is easy to see that
  $(\theta+k)P - P\theta \in V_Y^{k-1}(\Dsc_X)_p$.
  Write
  \[
    b(s,x) = s^m + a_1(x)s^{m-1} + \cdots + a_m(x). 
  \]
  Then we have
\begin{align*}
b(\theta + k,x)P - Pb(\theta,x)
&= (\theta+k)^m P - P\theta^m
+ a_1((\theta+k)^{m-1}P - P\theta^{m-1})
\\&\quad
+ (a_1P - Pa_1)\theta^{m-1}
+ \cdots + (a_mP - Pa_m)
\\&\equiv [a_1,P]\theta^{m-1} + \cdots + [a_m,P]
\quad \mod V_Y^{k-1}(\Dsc_X)_p.
\end{align*}
It follows that
\[
b(\theta+k,x)Pu \equiv ([a_1,P]\theta^{m-1} + \cdots + [a_m,P])u
= P'u
\quad \mod V_Y^{k-1}(\Dsc_X)_pu
\]
with $P' := [a_1,P]\theta^{m-1} + \cdots + [a_m,P]$.
Denote $\ord_x\, P$
the order of $P$ with respect to the derivations
$\partial = (\partial_1,\dots,\partial_n)$. 
Then $\ord_x\,P' < \ord_x\,P$ holds. 
Thus if $j \geq \ord_xP + 1$, then
\[
b(\theta+k,x)^jPu \in V_Y^{k-1}(\Dsc_X)_pu
\]
holds. This completes the proof.
\end{proof}

In general, for a complex
\[
X^\bullet : \quad \cdots {\longrightarrow} X^{k-1}
\stackrel{d_X^{k-1}}{\longrightarrow}
X^{k} \stackrel{d_X^k}{\longrightarrow}
X^{k+1} \stackrel{d_X^{k+1}}{\longrightarrow}
X^{k+2} \longrightarrow\cdots
\]
in an abelian category $\Csc$ and an integer $k$, define the truncated
complex by
\[
\tau^{\geq k}X^\bullet : \quad \cdots \rightarrow 0 \rightarrow 
  \mathrm{Coker}\, d_X^{k-1} \longrightarrow X^{k+1} 
  \stackrel{d_X^{k+1}}{\longrightarrow}
X^{k+2} \longrightarrow\cdots
\]
following the notation in \cite{KS}.
Then there is a chain map $X^\bullet \rightarrow \tau^{\geq k}X^\bullet$,
which induces an isomorphism
$H^j(X^\bullet) \stackrel{\sim}{\rightarrow} H^j(\tau^{\geq k}X^\bullet)$
for each $j \geq k$.
In particular, if $H^j(X^\bullet) = 0$ for all $j < k$,
the chain map above induces an isomorphism
$X^\bullet \stackrel{\sim}{\rightarrow} \tau^{\geq k}X^\bullet$
in the  derived category $D(\Csc)$. 

\begin{theorem}\label{th:inv}
  In the notation above, set $u_i = \varphi(e_i)$, where $e_1,\dots,e_{r_0}$
  are the unit vectors. Assume that there exists an indicial polynomial
  $b_i(s,x) \in (\Osc_Y)_p[s]$ of $u_i$ along $Y$ at $p$ for each $i$.
  Let $k_0 \leq k_1$ be integers (or $k_0 = -\infty$) such that
  $b_i(j,p)$ ($i=1,\dots,r_0$) do not vanish if an integer $j$
  satsifies $j \leq k_0$ or $j > k_1$.
  Let $V^\bullet$ be the complex
\begin{equation} \label{eq:truncation}
V^\bullet: \quad
  \cdots\rightarrow
  \frac{V_Y^{k_1}[\mvec_{d+1}](\Dsc_{Y\rightarrow X}^{r_{d+1}})_p}
     {V_Y^{k_0}[\mvec_{d+1}](\Dsc_{Y\rightarrow X}^{r_{d+1}})_p}
     \stackrel{\overline\psi_{d+1}}{\longrightarrow}
\cdots
\stackrel{\overline\psi_2}{\rightarrow}
\frac{V_Y^{k_1}[\mvec_1](\Dsc_{Y\rightarrow X}^{r_1})_p}
{V_Y^{k_0}[\mvec_1](\Dsc_{Y\rightarrow X}^{r_1})_p}
\stackrel{\overline\psi_1}{\rightarrow}
\frac{V_Y^{k_1}(\Dsc_{Y\rightarrow X}^{r})_p}
{V_Y^{k_0}(\Dsc_{Y\rightarrow X}^{r})_p}
\rightarrow 0
\end{equation}
  of free $(\Dsc_Y)_p$-modules of finite rank;
  here $\overline\psi_j$ is the homomorphism naturally induced by $\psi_j$
  and the denominators are $0$ if $k_0 = -\infty$. 
  Then $\mathbb{L}\iota^*(\Msc)_p$ is isomorphic to the bounded complex
  $\tau^{\geq -d}V^\bullet$   in the derived category
  $D^{-}(\mathrm{Mod}((\Dsc_Y)_p))$. 
  In particular, if $b_i(j,p)$ does not vanish for any $j \in \Z$ and
  any $i=1,\dots,r_0$, then $\mathbb{L}\iota^*(\Msc) = 0$ holds in
  $D^{-}(\mathrm{Mod}((\Dsc_Y)_p))$.
\end{theorem}

\begin{proof}
  Set $b(s,x) := b_1(s,x)\cdots b_{r_0}(s,x)$ and
 $\gr_Y^k(\Msc_p) := V_Y^k(\Msc_p)/V_Y^{k-1}(\Msc_p)$. 
It suffices to show that $K^\bullet[k](\{\gr_Y^j(\Msc_p)\},t)$ is exact
if $b(k,p) \neq 0$. 
  In view of (\ref{eq:Vresolution}), we have
  \[
   V_Y^k(\Msc_p) = V_Y^k(\Dsc_X)_pu_1 + \cdots + V_Y^k(\Dsc_X)_pu_{r_0}. 
 \]
 If $d=1$, then $K^\bullet[k](\{\gr_Y^j(\Msc_p)\},t)$ is the complex 
  \[
  0 \rightarrow \gr_Y^{k+1}(\Msc_p) \stackrel{t_1}{\longrightarrow}
  \gr_Y^{k}(\Msc_p) \rightarrow 0.
  \]
  Assume $v\in V_Y^{k+1}(\Msc_p)$ satisfies $t_1v \in V_Y^{k-1}(\Msc_p)$.
  By Lemma \ref{lemma:bsx}, $b(\theta+k+1,x)^jv$
  belongs to $V_Y^{k}(\Msc_p)$ for some integer $j$.
  Hence we have
  \[
b(\partial_{t_1}t_1+k,x)^jv = b(\theta+k+1,x)^jv \in V_Y^k(\Msc_p). 
  \]
  On the other hand, one has
  \[
b(\partial_{t_1}t_1+k,x)^j-b(k,x)^j = Qt_1
  \]
with some $Q \in V_Y^1(\Dsc_X)_p$. 
It follows that
$b(k,x)^jv$ belongs to $V_Y^k(\Msc_p)$. This implies
$v \in V_Y^k(\Msc_p)$ since $b(k,p) \neq 0$.

Next let $u$ be an arbitrary element of $V_Y^k(\Msc_p)$. 
Then $b(\theta+k,x)^ju$ belongs to $V_Y^{k-1}(\Msc_p)$ for some $j$.
There exists $Q \in V_Y^1(\Dsc_X)_p$ such that
\[
b(\theta+k,x)^j - b(k,x)^j = -t_1Q.
\]
Hence we have
\[
b(k,x)^ju \equiv t_1Qu \quad \mod V_Y^{k-1}(\Msc_p).
\]
Thus $t_1 : \gr_Y^{k+1}(\Msc_p) \rightarrow \gr_Y^k(\Msc_p)$
is an isomorphism.
We can show the claim for $d>1$ by induction using the fact that
$K^\bullet[k](\{\gr_Y^j(\Msc_p)\},t)$ is isomorphic to the mapping cone of
\[
K^\bullet[k+1](\{\gr_Y^j(\Msc_p)\},t_1,\dots,t_{d-1})
\stackrel{t_d}{\longrightarrow}
K^\bullet[k](\{\gr_Y^j(\Msc_p)\},t_1,\dots,t_{d-1}).
\]
See \cite{OTalgDmod} for details. 
This implies  $\mathbb{L}\iota^*(\Msc)_p$ is isomorphic to $V^\bullet$,
which is isomorphic to $\tau^{\geq -d}V^\bullet$ 
in the derived category since $H^j(\mathbb{L}\iota^*(\Msc)) = 0$ for $j < -d$. 
\end{proof}

\begin{example}\label{ex:Fuchs}\rm
  Let us consider a single equation $Pu=0$, i.e, a left $\Dsc_X$-module
  $\Msc = \Dsc_Xu = \Dsc_X/\Dsc_X P$, where
  \[
  P = t^k\partial_t^m + a_1(x)t^{k-1}\partial_t^{m-1}
  + \cdots + a_k(x)\partial_t^{m-k} + Q
  \]
  with $d=1$, $0 \leq k \leq m$, $a_i \in (\Osc_Y)_p$,
  and $Q \in V_Y^{m-k-1}(\Dsc_X)_p$. 
Then the indicial polynomial of $P$ along $Y$ at $p$ is
\begin{align*}
b(s,x) &= s(s-1)\cdots(s-m+1) + a_1(x) s(s-1)\cdots (s-m+2)
\\&\quad
+ \cdots + a_k(x)s(s-1)\cdots(s-m+k+1). 
\end{align*}
Note that $P$ is Fuchsian in the sense of \cite{BG} if
$\ord\,Q \leq m$. 
Assume $b(j,p) \neq 0$ for any integer $j \geq m-k$.
Then we can take $k_0 = -\infty$ and $k_1 = m-k-1$
in the notation of Theorem \ref{th:inv}. 
Hence $\mathbb{L}\iota^*(\Msc)$ is isomorphic to the complex
\[
0 \rightarrow F_Y^{-1}(\DscYX)_p \stackrel{P}{\longrightarrow}
F_Y^{m-k-1}(\DscYX)_p \rightarrow 0. 
\]
Since $F_Y^{-1}(\DscYX)_p = 0$, this implies an isomorphism
\[
\mathbb{L}\iota^*(\Msc)_p \simeq F_Y^{m-k-1}(\DscYX)_p
= \bigoplus_{i=0}^{m-k-1}(\Dsc_Y)_p\partial_t^i
\]
in $D^-(\mathrm{Mod}((\Dsc_Y)_p))$.
In particular, we have
\[
H^0(\mathbb{L}\iota^*(\Msc))_p \simeq (\Dsc_Y^{m-k})_p,
\qquad
H^{-1}(\mathbb{L}\iota^*(\Msc))_p =0
\]
as left $(\Dsc_Y)_p$-modules. 
\end{example}

\section{Algebraic formulation}\label{section:algebraic}

Let us consider the algebraic version of  $b$-functions
and indicial polynomials. 
We have only to treat $b$-functions and indicial polynomials along
a linear subspace by using the graph embedding.

In what follows we fix a field $K$ of characteristic zero,    
which we regard as the `base field' over which all the data are defined, 
and an algebraic closed field $\Omega$ which contains all fields 
we work with.  We also introduce an intermediate field $L$
such that $K \subset L \subset \Omega$ in order to clarify the
dependence of indicial polynomials and $b$-functions on the coefficient field. 
Most typically, we set $K = \Q$,  $\Omega = \C$
and let $L$ be an algebraic extension field of $\Q$. 

Let $X := K^{n+d}$ be the $(n+d)$-dimensional affine space over $K$
and let $(x,t) = (x_1,\dots,x_n,t_1,\dots,t_d)$ be its coordinates. 
Let $Y := X \times \{0\} \subset X \times K^d$ be the $n$-dimensional
affine subspace of $X$ with the coordinates $x$. 

Let us denote by $D_n$ the ring of differential operators with 
polynomial coefficients in $n$ variables $x$ over $K$,
and by $D_{n+d}$ that in $(n+d)$-variables $(x,t)$. 
The $V$-filtration on $D_{n+d}$ with respect to $Y$ is defined by
\[
V^i_Y(D_{n+d}) := \{ P \in D_{n+d} \mid P \langle t_1,\dots,t_d\rangle^j
\subset  \langle t_1,\dots,t_d)\rangle^{j-i} 
\quad (\forall j \in \Z) \} \quad (i\in\Z), 
\]
where $\langle t_1,\dots,t_d\rangle$ is the ideal of $K[t,x]$ generated by 
$t_1,\dots,t_d$. 
The associated graded ring is defined to be
\[
\gr_Y(D_{n+d}) := \bigoplus_{i\in\Z}\gr_Y^i(D_{n+d}), 
\qquad
\gr_Y^i(D_{n+d}) := V^i_Y(D_{n+d})/V_Y^{i-1}(D_{n+d}). 
\]

Let $M$ be a finitely generated left $D_{n+d}$-module 
and $u$ a nonzero element of $M$. 
The left ideal 
\[
\Ann_{D_{n+d}}u := \{P \in D_{n+d} \mid Pu = 0\}
\]
of $D_{n+d}$ is called the annihilator (ideal) of $u$. 
Then 
\[
\gr^i_Y(\Ann_{D_{n+d}}u) := (V^i_Y(D_{n+d})\cap \Ann_{D_{n+d}}u)/
(V^{i-1}_Y(D_{n+d})\cap \Ann_{D_{n+d}}u)
\qquad (i\in\Z)
\]
is a left $\gr_Y^0(D_{n+d})$-submodule of $\gr^i_Y(D_{n+d})$. 
The residue class of  $P \in V_Y^i(D_{n+d})$ belongs to
$\gr_Y^i(\Ann_{D_{n+d}}u)$
if and only if there exists $Q \in V_Y^{i-1}(D_{n+d})$ such that
  $(P+Q)u = 0$. 
Set 
$
\theta := t_1\partial_{t_1} + \cdots + t_d\partial_{t_d}
$ and let $s$ be the residue class of $\theta$ in $\gr_Y^0(D_{n+d})$. 
Then $\gr^0_Y(\Ann_{D_{n+d}}u) \cap K[s,x]$ is an ideal of the 
polynomial ring $K[s,x]$.  
We define an ideal $J_Y(u)$ of $K[s,x]$ by
\[
J_Y(u) := \gr^0_Y(\Ann_{D_{n+d}}u) \cap K[s,x]. 
\]
We remark that
\[
L\otimes_K J_Y(u)
= L\otimes_K \gr^0_Y(\Ann_{D_{n+d}}u) \cap L[s,x]
= \gr^0_Y(L\otimes_K\Ann_{D_{n+d}}u) \cap L[s,x]
\]
holds for any extension field $L$ of $K$. 

\begin{definition}\rm
  Let $p = (c_1,\dots,c_n)$  be a point of $\Omega^n$ and
  $L$ a subfield of $\Omega$. 
  Then we set
\[
Z_L(p) := \{f(x) \in L[x] \mid f(p) = 0\},
\]
which is a prime ideal of $L[x]$.
In particular, $Z_\Omega(p)$ is the maximal ideal of $\Omega[x]$
generated by $x_1-c_1,\dots,x_n-c_n$, and $Z_L(p) = Z_\Omega(p) \cap L[x]$. 
\end{definition}

For example, set $\Omega = \C$ and $p = (\alpha,\alpha\sqrt2) \in \C^2$ 
with a transcendental number $\alpha \in \C$. Then we have
\[
Z_{\Q}(p) = \langle 2x_1^2 - x_2^2\rangle, 
\quad
Z_{\Q(\sqrt2)}(p) = \langle \sqrt2 x_1 - x_2\rangle, 
\quad
Z_{\C}(p) = \langle x_1-\alpha,x_2-\sqrt2\alpha\rangle. 
\]


It is natural to specify a point $p$ of $\Omega^n$ in terms of 
the ideal $Z_L(p)$ rather than the coordinates of $p$. 
In fact, for any algebraic extension $L$ of $\Q$ and 
any prime ideal $\pp$ of $L[x]$, there exists 
a point $p$ of $\C^n$ such that $\pp = Z_L(p)$; 
see Proposition 1.4 of \cite{Mumford}.
That is, the map $\C^n \ni p \mapsto Z_L(p) \in \mathrm{Spec}\,L[x]$
is surjective. 
We shall see that an indicial polynomial at $p$ is
determined by the ideal $Z_L(p)$. 
However, we need a condition on $L$ for the $b$-function at $p$ to be
determined by $Z_L(p)$, in general.

\begin{definition}\label{def:algbfunction}\rm
Let $M$, $u$ be as above 
and $L$ be an extension field of $K$. 
Let $\pp$ be a prime ideal of $L[x]$. 
Set
$J_Y(u)_\pp := L[x]_\pp\otimes_{L[x]}L\otimes_K J_Y(u)$,
which is an ideal of $L[x]_\pp[s]$; 
here $L[x]_\pp$ denotes the localization of $L[x]$ with respect
to the multiplicative set $L[x]\setminus\pp$. 
Then the monic generator of the ideal $J_Y(u)_\pp \cap L[s]$ of $L[s]$,
if it is not $\{0\}$, 
is called the {\em $b$-function of $u$ along $Y$ at the prime ideal $\pp$},
which we denote by $b_\pp(s)$.
The monic generator of the ideal $J_Y(u) \cap K[s]$, if it is not $\{0\}$,  
is called the {\em global $b$-function} of $u$ along $Y$. 
\end{definition}

The condition $b_\pp(s) \in J_Y(u)_\pp$ is equivalent to
the existence of $a(x) \in L[x]\setminus \pp$ such that
$a(x)b_\pp(s) \in L\otimes_K J_Y(u)$. 
The correspondence $\pp \mapsto b_\pp(s)$ 
defines a map $\Spec L[x] \rightarrow L[s]$.
If $\pp_1$ and $\pp_2$ are prime ideals of $L[x]$ with $\pp_1 \subset \pp_2$, 
then $b_{\pp_1}(s)$ divides $b_{\pp_2}(s)$
since $L[x]_{\pp_2} \subset L[x]_{\pp_1}$.
The following fact is well-known for the classical Bernstein-Sato
polynomial. 

\begin{proposition} 
Let $L$ be an arbitrary extension field of $K$. 
The global $b$-function $b(s) \in K[x]$ of $u$ along $Y$ 
is the least common multiple of the $b$-functions of $u$
at the maximal ideals of $L[x]$. 
In particular, the global $b$-function exists if and only if
the $b$-function at every maximal ideal of $L[x]$ exists. 
\end{proposition}

\begin{proof}
It is easy to see (e.g., by the elimination algorithm) that 
\[
L\otimes_K J_Y(u) \cap L[s] = L\otimes_K(J_Y(u) \cap K[s]). 
\]
Let $b_\pp(s)$ be the $b$-function at a prime ideal $\pp$ of $L[x]$. 
Then $b_\pp(s) \in L[s]$  is a factor of $b(s)$ in $L[s]$
since $b(s) \in J_Y(u) \subset J_Y(u)_\pp$. 

Let $\tilde
b(s)$ be the 
least common multiple of the $b$-functions at the maximal ideals of $L[x]$. 
Then $\tilde b(s)$ is a factor of $b(s)$. 
Suppose $\tilde b(s) \neq b(s)$. Then $\tilde b(s)$ does not belong to 
$L\otimes_K J_Y(u)$. Set
\[
 I := (L\otimes_K J_Y(u) : \tilde b(s)) \cap L[x] 
 = \{a(x) \in L[x] \mid a(x)\tilde b(s) \in L\otimes_K J_Y(u)\},
\]
which is a proper ideal of $L[x]$ by the assumption. 
Let $\pp$ be a maximal ideal of $L[x]$ which contains $I$. 
There exists $a(x) \in L[x] \setminus \pp$ such that
$a(x)b_\pp(s) \in L\otimes_K J_Y(u)$.
Since $b_\pp(s)$ divides $\tilde b(s)$, 
we have $a(x) \in I \setminus \pp$, which is a contradiction. 
\end{proof}

The $b$-function essentially depends on the field extension $L$ of $K$
in general. 

\begin{example}\label{ex:sqrt2}\rm
  Set $n=d=1$ and consider $M := D_2/D_2(t\partial_t-x) + D_2(x^2-2))$
  with the base field $K = \Q$ with $u$ the residue class of $1$ in $M$.
  Then $J_Y(u)$ is the ideal of $\Q[s,x]$ generated by $s-x$ and $x^2-2$. 
  The global $b$-function of $u$ is $s^2-2$.
  If $L = \Q$ and $\pp = \langle x^2-2 \rangle \subset \Q[x]$, then
  the $b$-function of $u$ at $\pp$ is $s^2-2$.
  On the other hand, if $L$ contains $\Q(\sqrt2)$,
  the $b$-function at the prime ideal
  $\qq := \langle x-\sqrt2\rangle\subset L[x]$
  is $s-\sqrt2$ while $\qq \cap \Q[x] = \pp$.
  In fact, $(x+\sqrt2)(s-\sqrt2)$ belongs to $L\otimes_K J_Y(u)$. 
\end{example}

\begin{definition}\label{def:algindcial}\rm
Let $M$, $u$ be as above 
and $L$ be an extension field of $K$. 
Let $\pp$ be a prime ideal of $L[x]$.
Then a monic polynomial $b(s,x) \in L[x]_\pp[s]$ of the smallest degree in $s$
that belongs to 
$J_Y(u)_\pp$ 
is called an {\em (algebraic) indicial polynomial}
of $u$ along $Y$ at $\pp$.  
\end{definition}

\begin{example}\rm
Set $n=2$, $d=1$, and $M = D_3u$ with $(x_1t\dt - x_2)u = 0$. 
Then  we have $J_Y(u) = \langle x_1s - x_2\rangle$.
Let $\pp$ be a prime ideal of $K[x]$.
If $x_1 \not \in \pp$, the indicial polynomial of $u$ along $Y$ at $\pp$
is $s-x_2/x_1 \in K[x]_\pp[s]$.
For example, if $\pp = \langle x_2 \rangle$, then
$x_2/x_1 \in K[x]_\pp = \{ f/g \mid f,g \in K[x], g \not\in x_2K[x]\}$. 
If $x_1 \in\pp$, then there exists no indicial polynomial at $\pp$.
\end{example}

Contrary to the $b$-function, an indicial polynomial is stable under
field extension.

\begin{proposition}
  Let $b(s,x) \in K[x]_\pp[s]$ be an indicial polynomial of $u$ along $Y$ at
  a prime ideal $\pp$ of $K[x]$.  Let $\qq$ be a prime ideal
  of $L[x]$ with an extension field $L$ of $K$ such that
  $\qq \cap K[x] = \pp$. Then $b(s,x)$ can be regarded as an
  element of $L[x]_\qq[s]$ and as such it is an indicial polynomial
  of $u$ along $Y$ at $\qq$.
\end{proposition}

\begin{proof}
Let us define a filtration of $L[s,x]$ by
\[
F_k(L[s,x]) = \{ f(s,x) \in L[s,x] \mid
\mbox{$f(s,x)$ is of degree $\leq k$ in $s$ or $0$} \}.
\]
Set
\begin{align*}
  \gr_s^k(L\otimes_K J_Y(u)) &:= (F_k(L[s,x]) \cap L\otimes_K J_Y(u))
  /(F_{k-1}(L[s,x]) \cap L\otimes_K J_Y(u)),  
\\
I_k(L) &:= \{a(x) \in L[x] \mid a(x)s^k \in \gr_s^k(L\otimes_K J_Y(u))\}.
\end{align*}
  Then $I_k(L)$ is an increasing sequence of ideals of $L[x]$, hence 
  there exists $m \geq 0$ such that
  $I_k = I_m$ for any $k\geq m$. Set $I_\infty(L) := I_m(L)$. 

  The degree in $s$ of an indicial polynomial of $u$ at $\qq$ equals
  $k_0 := \min \{ k \mid I_k(L) \not\subset \qq\}$ if this set is not empty.
  In particular, an indicial polynomial at $\qq$ exists 
  if and only if $I_\infty(L) \not\subset \qq$. 
  Thus it suffices to show that 
  $I_k(L) \not\subset \qq$ holds if and only if
  $I_k(K) \not\subset \pp$.

  Note that $I_k(L) = L\otimes_K I_k(K)$ holds
  since $\gr_s^k(L\otimes_K J_Y(u)) = L\otimes_K\gr_s^k(J_Y(u))$. 
  Assume that there exists $a(x) \in I_k(L) \setminus \qq$.
  There exist $a_0(x) \in I_k(K)$,  and a finite number of $c_i \in L\setminus K$  
  and $a_i(x) \in I_k(K)$ with $i\geq 1$
  such that
\[
a(x) = a_0(x) + \sum_{i\geq 1} c_ia_i(x).
\]
  There is some $i \geq 0$ such that
  $a_i(x)$ does not belong to $\pp$;
  otherwise, $a(x)$ belongs to $\qq$.
  The converse implication is trivial. 
\end{proof}

\begin{example}\rm
  In the setting of Example \ref{ex:sqrt2},
Let $\pp$ be a prime ideal of  $L[x]$ with an arbitrary $L \supset K$.
If $\pp$ contains $x^2-2$, then $s-x$ is an indicial polynomial of
$u$ along $Y$ at $\pp$. Otherwise, the indicial polynomial is $1$.
\end{example}

As to the behavior of the $b$-function under field extension,
we have the following

\begin{proposition}
  Let $K \subset L \subset \tilde L$ be field extensions.
  Assume $\pp$ is a prime ideal of $L[x]$ and $\qq$ is a prime
  ideal of $\tilde L[x]$ such that $\qq \cap L[x] = \pp$.
  Let $b_\pp(s)$ be the $b$-function of $u$ at $\pp$ and
  assume that $b_\pp(s)$ decomposes into linear factors in $L[s]$.
  Then $b_\pp(s)$ is also the $b$-function of $u$ at $\qq$. 
\end{proposition}

\begin{proof}
  The $b$-function  $b_\qq(s)$ at $\qq$ divides $b_\pp(s)$ in $\tilde L[s]$.
  Hence $b_\qq(s)$ belongs to $L[s]$. There exists
  $a(x) \in \tilde L[x] \setminus \qq$ such that
  $a(x)b_\qq(s) \in \tilde L \otimes_K J_Y(u)$. 
Since we have
\[
a(x) \in (\tilde L \otimes_K J_Y(u) :b_\qq) \cap \tilde L[x]
= \tilde L\otimes_L ((L\otimes_K J_Y(u):b_\qq) \cap L[x]), 
\]
there exist a finite number of $c_i \in \tilde L$
and $a_i(x) \in (L\otimes_K J_Y(u):b_\qq) \cap L[x]$ with $c_0 = 1$ such that
\[
a(x) = a_0(x) + \sum_{i\geq 1}c_ia_i(x). 
\]
Then some $a_i(x)$ does not belong to $\pp$. 
Thus $(L\otimes_K J_Y(u) :b_\qq) \cap L[x]) \setminus \pp$ is not empty.
Consequently $b_\pp(s)$ divides $b_\qq(s)$.
\end{proof}

The folloing fact follows from the preceding two propositions
since $Z_L(p) = Z_\Omega(p) \cap L[x]$.
Recall that $\Omega$ is assumed to be algebraically closed. 

\begin{proposition}\label{prop:field_ext}
  Let $p$ be a point of $\Omega^n$ and let $L$ be a subfield of $\Omega$.
  \begin{enumerate}
  \item
    If the $b$-function $b(s)$ of $u$ at the prime ideal $Z_L(p)$ is a product
    of linear factors, then $b(s)$ is also the $b$-function of $u$
    at the maximal ideal $Z_\Omega(p)$.
  \item
    If an indicial polynomial $b(s,x)$ of $u$ at the prime ideal $Z_L(p)$
    exists, then $b(s,x)$ is also an indicial polynomial of $u$
    at the maximal ideal $Z_\Omega(p)$.
\end{enumerate}
\end{proposition}

Let us give an algebraic relation between the $b$-function and an indicial 
polynomial.

\begin{definition}\rm
For a prime ideal $\pp$ of $L[x]$, let $\kappa(\pp)$ be the
quotient field of the integral domain $L[x]/\pp$.
Then $\kappa(\pp)$ contains $L$ as a subfield. 
If $\pp$ is maximal, $\kappa(\pp) = L[x]/\pp$ is a finite
algebraic extension of $L$. 
There exists a natural ring homomorphism 
\[
L[x]_\pp \longrightarrow L[x]_\pp/L[x]_\pp\pp 
= \kappa(\pp),
\]
which is called the specialization to $\kappa(\pp)$. 
\end{definition}

\begin{example}\rm
If $\pp = \langle x_1^2-2 \rangle \subset \Q[x]$ with $n=2$, 
then $\kappa(\pp)$ is the field of rational functions
\[
\kappa(\pp) = (\Q[x_1]/\langle x_1^2-2 \rangle)(x_2)
\simeq \Q(\sqrt2)(x_2)
\]
and the specialization to $\kappa(\pp)$ is the substitution $x_1 = \sqrt 2$ or $x_1 = -\sqrt 2$ depending on
the choice from the two isomorphisms $\Q[x_1]/\langle x_1^2-2 \rangle \simeq
\Q(\sqrt2)$. 
\end{example}

\begin{proposition}
Let $\pp$ be a prime ideal of $L[x]$.
Asuume there exists an indicial polynomial $b(s,x) \in L[x]_\pp[s]$
of $u$ at $\pp$.  Then the specialization of $b(s,x)$ to $\kappa(\pp)[s]$
is uniquely determined. 
\end{proposition}  

\begin{proof}
Let $m$ be the degree of $b(s,x)$ in $s$.
Then by the same argument as the proof of Proposition 1,
in the local ring $L[x]_\pp$ with the maximal ideal
$L[x]_\pp\pp$ instead of the local ring $(\Osc_Y)_p$, 
we can show by induction that any $f(s,x)\in L\otimes_K J_Y(u)_\pp$ of
degree less than $m$ in $s$ belongs to $L[x]_\pp \pp [s]$. 
\end{proof}  

\begin{proposition}
Let $\pp$ be a prime ideal of $L[x]$.
\begin{enumerate}
\item
  If there exsts the $b$-function $b(s) \in L[s]$ of $u$ at $\pp$,
  there exists an indicial polynomial $\tilde b(s,x) \in L[x]_\pp[s]$
  of $u$ at $\pp$; moreover the spcecialization of
  $\tilde b(s,x)$ to $\kappa(\pp)$
  belongs to $\overline L[s]$ and divides $b(s)$ in $\overline L[s]$,
  where $\overline L$ is the algebraic closure of $L$. 
\item
  If there exists an indicial polynomial $\tilde b(s)$ of $u$ at $\pp$
  which belongs to $L[s]$, then $\tilde b(s)$ is the $b$-function of
  $u$ at $\pp$.
\end{enumerate}
\end{proposition}

\begin{proof}
  1. The existence of $\tilde b(s,x)$ immediately follows from that of $b(s)$. 
  Dividing $b(s)$ by $\tilde b(s,x)$ in $L[x]_\pp$, write
  \[
   b(s) = q(s,x)\tilde b(s,x) + r(s,x)
   \]
   with the degree of $r(s,x)$ in $s$ less than that of $\tilde b(s,x)$.
   Then $r(s,x)$ belongs to $L[x]_\pp\pp[s]$ by the preceding proposition.
   Thus the specialization $\tilde b(s)$
   of $\tilde b(s,x)$ to $\kappa(\pp)[s]$ divides
   $b(s)$  in $\kappa(\pp)[s]$. This implies $\tilde b(s) \in \overline L[s]$.
   The statement 2 follows from 1. 
\end{proof}  

Before concluding this section, let us give an algebraic formulation of
Theorem 1 concerning $b$-functions and indicial polynomials
with respect to coherent ideals. 
Let $M$ be a finitely generated $D_n$-module and $u$ a nonzero element of $M$. 
Let $F$ be an ideal of $K[x]$ and $f_1,\dots,f_d$ be a set of generators
of $F$.
Let $\iota : K^n \rightarrow K^n \times K^d$ be a polynomial map defined by
\[
\iota(x) = (x,f_1(x),\dots,f_d(x)) 
\]
and set $Z := \iota(K^n)$. 
Then the (algebraic) direct image $\iota_*(M)$ of $M$ with respect to $\iota$
is 
\[
\iota_*(M) := M \otimes_{K[x]}B_{Z|K^{n+d}}, 
\]
where
\[
B_{Z|K^{n+d}} := H^d_{[Z]}(K[x,t]) \simeq
K[x,t,h^{-1}]/\sum_{j=1}^d K[x,t,h_j^{-1}]
\]
with $h := (t_1-f_1)\cdots(t_d-f_d)$ and
$h_j := h/(t_j-f_j)$. 
Let $\delta(t_1-f_1)\cdots\delta(t_d-f_d)$ be the residue class of
$(t_1-f_1)^{-1}\cdots(t_d-f_d)^{-1}$ in $B_{Z|K^{n+d}}$ and set
\[
\iota_*(u) := u\otimes\delta(t_1-f_1)\cdots\delta(t_d-f_d).
\]

\begin{theorem}\label{th:BMSalg}
  Let $L$ be an extension field of $K$
  and let $\pp$ be a prime ideal of $L[x]$. 
  \begin{enumerate}
\item
    Let $b(s)$ be the $b$-function of $\iota_*(u)$ along $K^n \times \{0\}$
    at a prime ideal $\pp$ of $L[x]$.
    Then $b(s-d)$ does not depend on the choice of a set $f_1,\dots,f_d$
    of generators of $F$, which is called the $b$-function of $u$
    with respect to the ideal $F$ at $\pp$. 
  \item
    Let $b(s,x) \in L[x]_\pp[s]$
    be an indicial polynomial of $\iota_*(u)$ along
    $K^n\times\{0\}$ at $\pp$.
    Then $b(s-d,x)$ does not depend on the choice of a set $f_1,\dots,f_d$
    of generators of $F$, which is called an indicial polynomial of $u$
    with respect to $F$ at $\pp$. 
  \end{enumerate}
\end{theorem}

The proof is almost the same as that of Theorem 1.  
However, an algebraic formulation of Theorem 2 is not 
straightforward; one needs to algebraically 
flatten out the algebraic variety $Y$, 
which is not always possible, 
or to work with $D$-modules on $Y$ directly.
Nevertheless, 
it does not cause any problem concerning the computation of indicial polynomials
and $b$-functions along $Y$ in the analytic sense via {\em algebraic}
graph embedding as in Theorem 3, as will be certified in the next section.

\section{Comparison between analytic and algebraic  indicial polynomials}

In this section, we assume $K = \C$.
As in the preceding section, we denote  
$D_{n+d}$ the ring of differential operators
in $(x,t) = (x_1,\dots,x_n,t_1,\dots,t_d)$ with coefficients
in $\C[x,t]$, and denote $D_n$ the ring of differential operators
in $x = (x_1,\dots,x_n)$ with coefficients in $\C[x]$. 

Let $M$ be a finitely generated $D_{n+d}$-module and
let $u$ be a nonzero element of $M$. Set
\[
\Msc := \Dsc_X \otimes_{D_{n+d}}(D_{n+d}u)
\]
with $X := \C^{n+d}$. 
Since $\Dsc_X$ is flat over $D_{n+d}$, we can also regard $\Msc$
as a $\Dsc_X$-submodule of $\Dsc_X \otimes_{D_{n+d}}M$
generated by $1\otimes u$. 
Set $Y := \C^n\times \{0\} \subset \C^n \times \C^d$. 

\begin{theorem}\label{th:comparison}\rm
Set $\tilde u := 1\otimes u \in \Msc$ and 
let $p$ be an arbitrary  point of $Y$.
\begin{enumerate}
\item    
The $b$-function of $\tilde u$ 
along $Y$ at $p$ coincides with 
the $b$-function of $u$ along $Y$ at the maximal ideal $Z_\C(p)$
of $\C[x]$ corresponding to $p$. 
\item
  Let $b(s,x) \in \C[x]_{Z_\C(p)}[s]$ be
  an indicial polyomial of $u$ along $Y$ 
  at $Z_\C(p)$. 
  Then $b(s,x)$ is also an indicial polynomial of
  $\tilde u$ along $Y$ at $p$ via the natural inclusion
  $\C[x]_{Z_\C(p)} \subset (\Osc_Y)_p$. 

\end{enumerate}
\end{theorem}

\begin{proof} 
Since $\Dsc_X$  is flat over $D_{n+d}$, we have
\[
\Ann_{\Dsc_X}\tilde u = 
\Dsc_X\otimes_{D_{n+d}}\Ann_{D_{n+d}}u
= 
\Osc_X\otimes_{\C[x,t]}\Ann_{D_{n+d}}u. 
\]
As in Section \ref{section:inv}, 
for an integer vector $\mvec = (m_1,\dots,m_k) \in \Z^k$, 
we define the shifted $V$-filtration
\begin{align*}
V_Y^i[\mvec]((D_{n+d})^k) &:= V_Y^{i-m_1}(D_{n+d}) \oplus \cdots \oplus
V_Y^{i-m_k}(D_{n+d})
\end{align*}
on the free module $(D_{n+d})^k$.  
The associated graded module is defined to be
\begin{align*}
\gr_Y[\mvec]((D_{n+d})^k) &= \bigoplus_{i\in\Z}\gr_Y^i[\mvec]((D_{n+d})^k)
\end{align*}
with
\begin{align*} 
\gr_Y^i[\mvec]((D_{n+d})^k) &= V_Y^i[\mvec]((D_{n+d})^k)/V_Y^{i-1}[\mvec]((D_{n+d})^k).
\end{align*}
We omit $[\mvec]$ if $\mvec$ is the zero vector. 

There exist nonzero $P_1,\dots,P_k \in D_{n+d}$ such that 
\[
V_Y^i(D_{n+d}) \cap \Ann_{D_{n+d}} u  
= V_Y^{i-m_1}(D_{n+d})P_1 + \cdots + V_Y^{i-m_k}(D_{n+d})P_k 
\quad (\forall i\in\Z)
\]
with $m_j = \ord_Y(P_j) := \min\{i\in\Z \mid P_j \in V_Y^i(D_{n+d})\}$.
Moreover, there exist nonzrero $Q_1,\dots,Q_l \in (D_{n+d})^k$ such that 
\[
V_Y^i[\mvec']((D_{n+d})^l)
\stackrel{\psi}{\longrightarrow} 
V_Y^i[\mvec]((D_{n+d})^k)
\stackrel{\varphi}{\longrightarrow} V_Y^i(D_{n+d}) \cap \Ann_{D_{n+d}} u  
\rightarrow 0
\]
is exact for any integer $i$, where the homomorphisms $\varphi$ and $\psi$ 
are defined by 
\begin{align*}
\varphi(A_1,\dots,A_k) = A_1P_1 + \cdots + A_kP_k 
\quad (A_1,\dots,A_k \in D_{n+d}), 
\\
\psi(B_1,\dots,B_l) = B_1Q_1 + \cdots + B_lQ_l 
\quad (B_1,\dots,B_l \in D_{n+d})
\end{align*}
with $\mvec' := (\ord_Y(Q_1),\dots,\ord_Y(Q_l))$. 
In fact, $P_i$ and $Q_j$ can be computed as Gr\"obner bases with respect to 
orderings adapted to the shifted $V$-filtrations;  
see \cite{OTalgDmod} for details. 
In particular, the induced sequence
\begin{equation}\label{eq:exactAlg}
\gr_Y^i[\mvec']((D_{n+d})^l)
\stackrel{\overline\psi}{\longrightarrow} 
\gr_Y^i[\mvec]((D_{n+d})^k)
\stackrel{\overline\varphi}{\longrightarrow} \gr_Y^i(\Ann_{D_{n+d}} u)  
\rightarrow 0
\end{equation}
is exact with 
\[
\gr_V^i(\Ann_{D_{n+d}} u) := (V_Y^i(D_{n+d})\cap \Ann_{D_{n+d}} u)
/(V_Y^{i-1}(D_{n+d})\cap \Ann_{D_{n+d}} u)
.
\]

Let us show that the stalk of the sheaf
\[
\gr_Y^0(\Ann_{\Dsc_X} \tilde u) 
:= (V_Y^0(\Dsc_X)\cap \Ann_{\Dsc_X} \tilde u)
/(V_Y^{-1}(\Dsc_X)\cap \Ann_{\Dsc_X} \tilde u)
\subset \gr_Y^0(\Dsc_X)
\]
at $p$ is generated by $\gr_Y^0(\Ann_{D_{n+d}} u)$ 
over $(\Osc_X)_p$, i.e,
\begin{equation}\label{eq:gaga_gr}
\gr_Y^0(\Ann_{\Dsc_X} \tilde u)_p 
=  (\Osc_Y)_{p}\otimes_{\C[x]}\gr_Y^0(\Ann_{D_{n+d}}u) 
\end{equation}
holds.
Since $(\Osc_Y)_{p}$ is flat over $\C[x]$, 
(\ref{eq:exactAlg}) yields the exact sequence
\begin{equation}\label{eq:exactAn}
\gr_Y^i[\mvec']((\Dsc_X)^l)_{p}
\stackrel{\tilde\psi}{\longrightarrow} 
\gr_Y^i[\mvec]((\Dsc_X)^k)_{p}
\stackrel{\tilde\varphi}{\longrightarrow}
(\Osc_Y)_{p}\otimes_{\C[x]}\gr_Y^i(\Ann_{D_{n+d}} u)  
\rightarrow 0. 
\end{equation}
Assume $P$ belongs to 
$(V_Y^0(\Dsc_X)_{p}\setminus V_Y^{-1}(\Dsc_X)_{p})
\cap \Ann_{(\Dsc_X)_{p}}\tilde u$. 
Since $\Ann_{(\Dsc_X)_{p}}\tilde u$ is generated by 
$\Ann_{D_{n+d}}u$ over $(\Dsc_X)_{p}$, there exist 
$A_1,\dots,A_k \in (\Dsc_X)_{p}$ such that
\[
P = A_1P_1 + \cdots + A_kP_k. 
\]
Set $i := \max\{\ord_Y(A_jQ_j) \mid 1 \leq j \leq k\}$. 
If $i=0$, then the residue class $\sigma_Y(P)$ of $P$ in $\gr_Y^0(\Dsc_X)$
belongs to 
$(\Osc_Y)_{p}\otimes_{\C[x]}\gr_Y^0(\Ann_{D_{n+d}}v)$ and we are done. 
Suppose $i > 0$. Then we have
\[
\sigma_Y^{i-m_1}(A_1)\sigma_Y(P_1) +\cdots+ 
\sigma_Y^{i-m_k}(A_k)\sigma_Y(P_k) = 0,  
\]
where we set $\sigma_Y^i(A) = \sigma_Y(A)$ if $\ord_Y A = i$
and $\sigma_Y^i(A) = 0$ if $\ord_Y A < i$. 
In view of the exact sequence (\ref{eq:exactAn}), 
there exist $U_1,\dots,U_l \in (\Dsc_X)_{p}$ such that 
$\ord_Y(U_j) = i-m'_j$ and 
\[
(B_1,\dots,B_k) := 
(A_1,\dots,A_k) - \sum_{j=1}^l U_jQ_j \,\in\, V_Y^{i-1}[\mvec](\Dsc_X^k)_{p}. 
\] 
Hence we get another expression
\[
P = B_1P_1 + \cdots + B_kP_k, 
\qquad
\ord_Y(B_jQ_j) \leq i-1 \quad (1\leq j \leq k).
\]
Proceeding inductively we obtain $A_j \in V_Y^{-m_j}(\Dsc_X)_{p}$ such that
\[
P = A_1P_1 + \cdots + A_kP_k. 
\]
This implies (\ref{eq:gaga_gr}).
Next let $s$ be the residue class of
$\theta = t_1\partial_{t_1}+\cdots +t_d\partial_{t_d}$ in
$\gr_Y^0(\Dsc_X)$ or $\gr_Y^0(D_{n+d})$ and set
\begin{align*}
\Jsc_Y(\tilde u) &:= \gr_Y^0(\Ann_{\Dsc_X} \tilde u) \cap \Osc_Y[s],
\quad
J_Y(u) := \gr_Y^0(\Ann_{D_{n+d}}u) \cap \C[s,x]. 
\end{align*}
In the same argument as above, with respect to the
order of differential operator instead of $\ord_Y$, we can show that 
\begin{equation}\label{eq:gaga_J}
  \Jsc_Y(\tilde u)_p = (\Osc_Y)_p\otimes_{\C[x]} J_Y(u)
\end{equation}
holds by virtue of (\ref{eq:gaga_gr}). 
Let us define filtrations on $\Osc_Y[s]$ and on $\C[s,x]$ by
\begin{align*}
F_s^k(\Osc_Y[s]) &:= \{f(s,x)\in \Osc_Y[s] \mid \deg_s f(s,x) \leq k\},
\\
F_s^k(\C[s,x]) &:= \{f(s,x)\in \C[s,x] \mid \deg_s f(s,x) \leq k\},
\end{align*}
where $\deg_s$ denotes the degree in $s$. Then
\[
\gr_s^k(\Osc_Y[s]) := F_s^k(\Osc_Y[s])/F_s^{k-1}(\Osc_Y[s]),
\quad
\gr_s^k(\C[s,x]) := F_s^k(\C[s,x])/F_s^{k-1}(\C[s,x])
\]
are free modules of rank one over $\Osc_Y$ and $\C[x]$
respectively. 
Define submodules 
\begin{align*}
\gr_s^k(\Jsc_Y(\tilde u)) &:=
(F_s^k(\Osc_Y[s]) \cap \Jsc_Y(\tilde u))/
(F_s^{k-1}(\Osc_Y[s]) \cap \Jsc_Y(\tilde u)),
\\
\gr_s^k(J_Y(u)) &:= (F_s^k(\C[s,x]) \cap J_Y(u))
/(F_s^{k-1}(\C[s,x]) \cap J_Y(u))
\end{align*}
of $\gr_s^k(\Osc_Y[s])$ and of $\gr_s^k(\C[s,x])$ respectively. 
Note that
\[
\gr_s^k(\Jsc_Y(\tilde u))_p = (\Osc_Y)_p\otimes_{\C[x]}\gr_s^k(J_Y(u))
\]
holds in view of (\ref{eq:gaga_J}).

The degree of an indicial polynomial of $\tilde u$ 
at $p$ is the minimum integer $k$ such that
  $\gr_s^k(\Jsc_Y(\tilde u)) = \gr_s^k(\Osc_Y[s])$. 
  On the other hand, the degree of an indicial polynomial of
  $u$ at $\pp := Z_{\C}(p)$  is 
  the minimum integer $k$ such that
\[
\C[x]_\pp\otimes_{\C[x]}\gr_s^k(J_Y(u))
= \C[x]_\pp\otimes_{\C[x]}\gr_s^k(\C[s,x]) .
\]
These two conditions are equivalent 
since $(\Osc_Y)_p$ is faithfully flat over $\C[x]_\pp$ and
\[
(\Osc_Y)_p\otimes_{\C[x]_\pp}\C[x]_\pp\otimes_{\C[x]}\gr_s^k(J_Y(u)) =
(\Osc_Y)_p\otimes_{\C[x]}\gr_s^k(J_Y(u)) = \gr_s^k(\Jsc_Y(\tilde u)).
\]
Thus the degree of an indicial polynomial of $\tilde u$ at $p$
  is the same as the degree of an indicial polynomial of $u$ at $\pp$.
  This implies that $b(s,x)$ is also an indicial polynomial of $\tilde u$
  at $p$. 

  The coincidence of the $b$-function of $\tilde u$ at $p$ and that of
  $u$ at $\pp$ follows from
  \[
  \Jsc_Y(\tilde u)_p \cap \C[s]
  = (\Osc_Y)_p \otimes_{\C[x]}J_Y(u) \cap \C[s]
  = \C[x]_\pp\otimes_{\C[x]}J_Y(u) \cap \C[s], 
    \]
which is also a consequence of the faithfull flatness
of $(\Osc_Y)_p$ over $\C[x]_\pp$.
In fact, if $b$ belongs to $\Jsc_Y(u)_p\cap \C[s]$,
then we have
\begin{align*}
0 &= (\Jsc_Y(u)_p + (\Osc_Y)_p b)/\Jsc_Y(u)_p
\\
&= (\Osc_Y)_p\otimes_{\C[x]_\pp} ((\C[x]_\pp\otimes_{\C[x]}J_Y(u)
  + \C[x]_\pp b)/\C[x]_\pp\otimes_{\C[x]}J_Y(u)). 
\end{align*}
This implies $b \in \C[x]_\pp\otimes_{\C[x]}J_Y(u)$. 
\end{proof}

\section{Algorithms}\label{section:algorithm}

\subsection{Computation of the graph embedding}

Let $X := K^n$ be the $n$-dimensional affine space over a field $K$
of characteristic $0$ 
and let $x = (x_1,\dots,x_n)$ be its coordinates. 
Let us denote by $D_n$ the ring of differential operators with 
polynomial coefficients in $n$ variables $x$ over $K$.
Let $t = (t_1,\dots,t_d)$ be the coordinates of the affine space $K^d$. 
We denote $D_{n+d}$ the ring of differential operators
in $(n+d)$-variables $(x,t)$. 

Let $f_1,\dots,f_d$ be polynomials in $K[x]$
and let $\iota$ be the embedding defined by
\[
\iota\,:\, K^n \ni x \longmapsto (x,f_1(x),\dots,f_d(x)) \in K^n \times K^d.  
\]
For a $D$-module $M = Du$, the direct image $\iota_*(M)$ can be computed 
by the following algorithm, which is essentially due to Walther \cite{Walther}. 

\begin{algorithm}\label{alg:embedding}\rm
  Input: A set $G$ of generators of the left ideal $\Ann_{D_n}u$ of $D_n$.
  \\
  Output: A set $G'$ of generators of the left ideal $\Ann_{D_{n+d}}\iota_*(u)$
  of $D_{n+d}$.
\begin{enumerate}
\item    
For $P = P(x,\partial_{x_1},\dots,\partial_{x_n}) \in G$, set
\[
\tau(P,f_1,\dots,f_d) := P\left(x, \,
\partial_{x_1} + \sum_{j=1}^d \frac{\partial f_j}{\partial{x_1}}
\partial_{t_j},
\dots,
\partial_{x_n} + \sum_{j=1}^d \frac{\partial f_j}{\partial{x_n}}
\partial_{t_j}
\right) .
\]
This substitution is well-defined in the ring $D_{n+d}$ 
since the operators which are substituted for $\partial_{x_1},\dots,
\partial_{x_n}$ commute with one another. 
\item
Set
\[  
  G' := \{\tau(P,f_1,\dots,f_d) \mid P \in G \} \cup 
\{ t_j  - f_j(x) \mid j=1,\dots,d\} . 
\]
\end{enumerate}
\end{algorithm}

Let us show the correctness.
Let $I$ be the left ideal of $D_{n+d}$ generated by $G'$.
The inclusion $I \subset \Ann_{D_{n+d}}\iota_*(u)$ follows from
the equality
\begin{align*}
\left(\partial_{x_i} + \sum_{j=1}^d \frac{\partial f_j}{\partial x_i}
\partial_{t_j} \right)\iota_*(u)
&= 
\left(\partial_{x_i} + \sum_{j=1}^d \frac{\partial f_j}{\partial x_i}
\partial_{t_j} \right) (u\otimes\delta(t_1-f_1)\cdots\delta(t_d-f_d))
\\ &
= (\partial_{x_i}u)\otimes \delta(t_1-f_1)\cdots\delta(t_d-f_d). 
\end{align*}
Conversely, suppose that $P \in D_{n+p}$ annihilates $\iota_*(u)$. 
We can rewrite $P$ in the form
\begin{multline*}
P = \sum_{\alpha\in\N^n,\nu\in\N^d}
p_{\alpha,\nu}(x) \partial_t^\nu
\Bigl(\partial_{x_1} + \sum_{j=1}^d \frac{\partial f_j}{\partial{x_1}}
\partial_{t_j} \Bigr)^{\alpha_1} 
\cdots
\Bigl(\partial_{x_n} + \sum_{j=1}^d \frac{\partial f_j}{\partial{x_n}}
\partial_{t_j} \Bigr)^{\alpha_n} 
\\
+ \sum_{j=1}^d Q_j\cdot(t_j-f_j(x))
\end{multline*}
with $p_{\alpha,\nu}(x) \in  K[x]$ and $Q_j \in D_{n+p}$. 
Setting $P_\nu := \sum_{\alpha \in \N^n}p_{\alpha,\nu}(x)\partial_x^\alpha$,
we get
\begin{align*}
0 &= P(u\otimes\delta(t_1-f_1)\cdots\delta(t_d-f_d)) 
= \sum_{\nu\in\N^d}\partial_t^\nu \tau(P_\nu,f_1,\dots,f_d)
(u\otimes \delta(t_1-f_1)\cdots\delta(t_d-f_d))
\\&
= \sum_{\nu\in\N^d}P_\nu u\otimes \partial_t^\nu \delta(t_1-f_1)\cdots\delta(t_d-f_d).  
\end{align*}
It follows that $P_\nu u = 0$ holds for each $\nu$. 
Thus we have
\[
 P = \sum_{\nu \in \N^d} \partial_t^\nu \tau(P_\nu,f_1,\dots,f_d) 
+ \sum_{j=1}^d Q_j\cdot(t_j-f_j(x))\,\, \in J. 
\]

\subsection{Computation of an indicial polynomial}

Following the notation in Section \ref{section:algebraic}, 
let $X := K^{n+d}$ be the $(n+d)$-dimensional affine space over $K$
and let $(x,t) = (x_1,\dots,x_n,t_1,\dots,t_d)$ be its coordinates. 
Let $Y := X \times \{0\} \subset X \times K^d$ be the $n$-dimensional
affine subspace of $X$ with the coordinates $x$. 

Let us denote by $D_n$ the ring of differential operators with 
polynomial coefficients in $n$ variables $x$ over $K$,
and by $D_{n+d}$ that in $(n+d)$-variables $(x,t)$.

First let us give an algorithm for computing algebraic indicial polynomials. 
Let $K$ be the base field, most typically $\Q$, over which the data are
defined. 
Let $u$ be a nonzero element of a finitely generated $D_{n+d}$-module $M$
as in Section \ref{section:algebraic}. 
First we compute the ideal
$J_Y(u) := K[s,x] \cap \gr_Y(\Ann_{D_{n+d}}u)$ 
by the first two steps of Algorithm 4.6 of \cite{OTalgDmod}.
An outline of the algorithm is as follows: 

\begin{algorithm}\label{alg:Jideal}\rm 
  Input: a set of generators of the left ideal $\Ann_{D_{n+d}}u$ of $D_{n+d}$.
  \\
  Output: a set of generators of the ideal $J_Y(u)$ of $K[s,x]$.   
\begin{enumerate}
\item
  Compute a Gr\"obner basis of $\Ann_{D_n+d}u$ which
  is adapted to the $V_Y$-filtration and extract a set of generators
  of $\gr_Y^0(\Ann_{D_{n+d}}u)$.
\item
  Compute the intersection $I:= \gr_Y^0(\Ann_{D_{n+d}}u)
  \cap K[t_1\partial_{t_1},\dots,t_d\partial_{t_d}]\otimes_K D_n$ 
  by using  multi-homogenization with respect to the pairs
  $(t_1,\partial_{t_1}),\dots,(t_d,\partial_{t_d})$.
\item
  Compute the intersection $J_Y(u) = I \cap K[\theta,x]$
  by elimination.
\end{enumerate}    
\end{algorithm}
  
Now that a set of generators of $J_Y(u)$ is obtained, the rest of
the computations are done in polynomial rings. 

\begin{algorithm}\rm \label{alg:indicial}
  Input: a set of generators of the ideal $J_Y(u)$ of $K[s,x]$ and
  a set of generators of a prime ideal $\pp$ of $K[x]$.
\\
Output: an indicial polynomial $b(s,x) \in K[x]_\pp[s]$
of $u$ along $Y$ at $\pp$ of $K[x]$ if there is one. 
\begin{enumerate}
\item
  Compute a reduced Gr\"obner basis $G = \{f_1,\dots,f_m\}$ of $J_Y(u)$
    with respect to a term order  $\prec$ such that $x_1,\dots,x_n \prec s$.
    Let $a_i \in K[x]$ be the leading coefficient of $f_i$ with respect to $s$. 
\item
  Set   $N := \max\{\deg_s f_i \mid i=1,\dots,m\}$.  
  For each $k = 0,\dots,N$, let $I_k$ be the ideal of $K[x]$ generated by
  $\{a_i \mid \deg_s f_i \leq k \,(i=1,\dots,m)\}$. 
\item
 If $I_N \not\subset\pp$, set
  $k_0 := \min\{k \mid I_k \not\subset \pp \}$.
  Then there exists $f_i \in G$ such that $\deg_s f_i = k$
  and $a_i \not\in \pp$. Set $b(s,x) := f_i(s,x)/a_i(x)$. 
  If $I_N \subset \pp$, then there is no indicial polynomial at $\pp$. 
\end{enumerate}
\end{algorithm}

If $K$ is a subfield of $\C$, this algorithm provides a stratification of $\C^n$
with respect to the degree of indicial polynomials.
In the notation of the preceding algorithm, set
\[
V_k := \mathbb{V}(I_k) = \{p \in \C^n \mid
\mbox{$a(p) = 0$ for any $a \in I_k$} \}. 
\]
Then $V_k$ is a decreasing sequence of alebraic sets of $\C^n$. 
An indicial polynomial at the maximal ideal $Z_\C(p)$
exists and its degree is $k$ for any $p \in V_{k-1} \setminus V_{k}$
with $k = 0,1,\dots,N$; here we set $V_{-1} = \C^n$.
Moreover, if we choose  $f_i \in G$ such that $\deg_s f_i = k$,
then $f_i/a_i$ is an indicial polynomial at $Z_\C(p)$ for any
$p \in V_{k-1} \setminus V_k$ such that $a_i(p) \neq 0$. 
Note that we do not need field extension in the computations above. 

\subsection{Computation of the $b$-function}

Let us review existing algorithms for computing $b$-functions.
First in \cite{OakuDuke}, 
an algorithm  (Algorithm 4.5) to compute the $b$-function
of a $D$-module along a hyperplane at a point of $K^n$ was given.
This algorithm is based on the fact that
the roots of the $b$-function at a point $p$
coincides with the zeros of the ideal $J_Y(u)(\overline K)|_{x=p}$
of $\overline K[s]$
(Thoerem 4.6 of \cite{OakuDuke}), which also follows from Propositions 1 and 2
in Section 1  
and does not require the (existence of) global $b$-function. 
However, it requires factorization of a univariate polynomial into linear factors in general, 
which in fact is unnecessary for the classical Bernstein-Sato polynomial
by virtue of Kashiwara's theorem \cite{KashiwaraBfunction}. 

Secondly, an algorithm for computing $b$-functions of a $D$-module
along an affine subspace of arbitrary codimension
with a corresponding stratification was proposed
in \cite{OTalgDmod} (Algorithm 4.6).
This algorithm makes use of a primary decomposition of the ideal
$L\otimes_K J_Y(u)$ for a stratification. 
In order to compute the $b$-function at a point, one can use the
method of \cite{OakuDuke} mentioned above. 
Note also that the algorithm in \cite{OTalgDmod} can be readily 
applied to the computation of the Bernstein-Sato polynomial of an
arbitrary variety in the sense of Budur-Musta\c t\~a-Saito \cite{BMS};
in fact, Example 4.7 in \cite{OTalgDmod} is the computation of
the Bernstein-Sato polynomial of the variety
defined by the ideal $\langle x^3-y^2, y^3-z^2\rangle$. 

Thirdly, Nishiyama and Noro \cite{NN} introduced a much
more efficient algorithm to compute a stratification
with respect to the Bernstein-Sato polynomial without primary decomposition. 
In fact, their method also applies to the $b$-function of a $D$-module   
if the global $b$-function exists and its factorization into linear factors 
in the splitting field is obtained.

Our algorithm for indicial polynomials provides a heuristic method for
computing the $b$-function with an associated stratification
without field extension, the global $b$-function, or primary decomposition,
especially for holonomic $D$-modules including the case with Bernstein-Sato
polynomials (see Example \ref{ex:BMS} below). 

\subsection{Case with parameters}

Let $y = (y_1,\dots,y_m)$ be indeterminates 
and $M = D_{n+d}[y]u$ be a $D_{n+d}[y]$-module generated by $u \in M$. 
For a prime ideal $\qq$ of $K[x,y]$, 
let  $b(s,x,y) \in K[x,y]_\qq[s]$ be an indicial polynomial
of $u$ along $Z :=  K^n \times K^d \times \{0\}
= \{(x,y,t) \in K^{n+m+d} \mid t=0\}$
at $\qq$ according to Definition \ref{def:indicial}. 
Then there exists $P \in K[x,y]_\qq\otimes_{K[x,y]}V_Z^{-1}(D_{n+m+d})$
such that
\[
(b(\theta,x,y) + P)u = 0
\]
holds in $K[x,y]_\pp\otimes_{K[x,y]}M$. 
It is easy to see that we can choose $P$ so that
$P \in K[x,y]_\qq\otimes_{K[x,y]}D_{n+d}[y]$, i.e.,
$P$ does not contain derivations with respect to $y$
in view of Algorithms \ref{alg:Jideal} and \ref{alg:indicial}.

Set $L := \kappa(\qq \cap K[y])$ and 
let $\rho : K[x,y] \rightarrow L[x]$
be the homomorphism defined by the specialization.
Suppose $\pp$ is a prime ideal of $L[x]$
such that $\pp \cap K[x] = \qq \cap K[x]$. 
Then $\rho$ induces a homomorphism
$
\rho_\qq : K[x,y]_\qq \rightarrow L[x]_\pp$. 
Set $b'(s,x) := \rho_\qq(b(s,x,y)) \in L[x]$.
Then 
\[
(b'(\theta,x) + \rho_\qq(P))u = 0  
\]
holds in $L[x]_\pp\otimes_{K[x,y]}M$.
For example, if $n=d=m=1$, $K = \Q$, and $\qq = \langle x,y^2-2\rangle$,
then we have $L = \Q[y]/\langle y^2-2\rangle \simeq \Q(\sqrt 2)$
and $\pp = L[x]x$. 
If $K$ is algebraically closed, it would be natural to take
a maximal ideal $\pp$ of $K[x]$ and a maximal ideal $\pp'$ of $K[y]$ and
let $\qq$ be the maximal ideal of $K[x,y]$ generated by $\pp$ and $\pp'$. 
Then the homomorphism $\rho_\qq: K[x,y]_\qq \rightarrow K[x]_\pp$
is simply the substitution $y = p'$, where $p'$ is the point of $K^m$
corresponding to $\pp'$ (see Example \ref{ex:AHG}).

\subsection{An operator associated with an indicial polynomial}

Once an indicial polynomial $b(s,x)$ of $u$ along $Y$ at a
prime ideal $\pp$ of $K[x]$ is obtained,
it is also important to find a differential operator  $P$ which belongs to
$K[x]_\pp\otimes_{K[x]}V_Y^{-1}(D_{n+d})$ and satisfies
\begin{equation}\label{eq:bP}
(b(\theta,x) + P)u = 0.
\end{equation}
Let $\prec$ be an ordering for the monomials in $D_{n+d}$ 
which is compatible with the filtration $\{V_Y^j(D_{n+d})\}$
and induces a well-ordering among the monomials in
$V_Y^j(D_{n+d}) \setminus V_Y^{j-1}(D_{n+d})$ for each $j$. 
Let $P_1,\dots,P_k$ be a Gr\"obner basis of $I := \Ann_{D_{n+d}}u$
with respect to $\prec$, which can be computed by the Buchberger algorithm
with homogenization although $\prec$ is not a well-ordering
(see \cite{OakuDuke}, \cite{OTalgDmod}). 
There exists $a(x) \in K[x]\setminus\pp$ such that $a(x)b(s,x) \in J_Y(u)$. 
By the (partial) division algorithm, we can find
$Q_1,\dots,Q_k,R \in D_{n+d}$ such that
\[
a(x)b(\theta,x) = Q_1P_1 + \cdots + Q_kP_k + R
\]
with $Q_iP_i \in V_Y^0(D_{n+d})$, $R \in V_Y^{-1}(D_{n+d})$. 
Then $(a(x)b(\theta,x)-R)u = 0$ holds. 
Hence we can put $P = -R/a$ in (\ref{eq:bP}). 

However, to find such $P$ of the smallest possible order is essential 
to the study of the solutions of the system of
differential equations corresponding to $u$. 
In particular, the system is called (weakly) regular or Fuchsian along $Y$
if we can take such $P$ of order less than or equal to the degree
of $b(s,x'')$ in $s$. 
Even in this case, it is not reasonable to assume $b(s,x)$ is an indicial
polynomial, that is, a monomic polynomial in $J_Y(u)_\pp \subset K[x]_\pp[s]$
of the smallest degree.

\begin{example}\rm
  Set $n = d=1$ and 
  $M = D_{n+d}u$ with $\partial_t^2u = (\partial_t + \partial_x^2)u = 0$. 
Then the indicial polynomial of $u$ along $Y$ at any maximal ideal $\pp$   
of $K[x]$ is $s$, but its associated operator is $\partial_x^2$.
On the other hand, $\theta(\theta-1)u = 0$ holds with the associated operator
$0$. 
\end{example}
In case $\pp$ is a maximal ideal,
the present author \cite{OakuRegular}
proposed an algorithm to compute, and to determine whether there exist, 
$b(s,x) \in J_Y(u)_\pp$ of the smallest degree and 
$P \in V_Y^{-1}(D_{n+d})$ with $\ord\, P \leq \deg_sb(s,x)$
for which (\ref{eq:bP}) is satisfied. 

\subsection{Computation of the inverse image}

Set $X := K^{n+d} \ni (x,t)$, $Y := K^n \times \{0\} \subset K^{n+d}$
and $\iota : Y \rightarrow X$ be the natural embedding. 
Set
\[
\DYX := D_{n+d}/(t_1D_{n+d} + \cdots + t_dD_{n+d}),
\]
which is a $(D_n,D_{n+d})$-bimodule. 
Let $M$ be a finitely generated left $D_{n+d}$-module. 
Then the (algebraic) inverse image $\mathbb{L}\iota^*(M)$ of $M$ with respect to $\iota$
is defined by
\[
\mathbb{L} \iota^*(M) := \DYX \stackrel{\mathbb{L}}{\otimes}_{D_{n+d}}M 
\]
in the derivied category $D^-(\mathrm{Mod}(D_n))$. 
It is related to the analytic version of the inverse image defined in
Section \ref{section:inv} by
\[
\Dsc_Y\otimes_{D_n}\mathbb{L}\iota^*(M) \simeq
\mathbb{L}\iota^*(\Dsc_X\otimes_{D_{n+d}}M)
\]
if $K$ is a subfield of $\C$. 
Let us state a generalization of the algorithm for $\mathbb{L}\iota^*(M)$
given in \cite{OakuAAM}, \cite{OTalgDmod} by using the $b$-function, which is
an algebraic version of Theorem \ref{th:inv}.
However, a little care must be taken if we replace the $b$-function by
an indicial polynomial; we need localization in general. 

First construct a (partial) free resolution 
\begin{equation}\label{eq:algresolution}
  (D_{n+d}^{r_{d+1}})_p
  \stackrel{\psi_{d+1}}{\longrightarrow} 
 (D_{n+d}^{r_{d}})_p 
\stackrel{\psi_{d}}{\longrightarrow} 
\cdots 
\stackrel{\psi_2}{\longrightarrow} (D_{n+d}^{r_1})_p
\stackrel{\psi_1}{\longrightarrow} (D_{n+d}^{r_0})_p
\stackrel{\varphi}{\longrightarrow}M \longrightarrow 0
\end{equation}
of  $M$  together with shift vectors 
$\mvec_1 \in \Z^{r_1}$, $...$, $\mvec_l \in \Z^{r_l}$ such that 
\[
\psi_{j+1}(V_Y^k[\mvec_{j+1}](D_{n+d}^{r_{j+1}}))
\subset V_Y^k[\mvec_j](D_{n+d}^{r_j})
\] 
holds for $j\geq 0$ with $\mvec_0$ the zero vector and that 
(\ref{eq:algresolution}) induces an exact sequence
\[
V_Y^k[\mvec_{d+1}](\Dsc_X^{r_{d+1}})
\stackrel{\psi_{d+1}}{\longrightarrow} 
V_Y^k[\mvec_d](\Dsc_X^{r_{d}})
\stackrel{\psi_{d}}{\longrightarrow} 
\cdots
\stackrel{\psi_1}{\longrightarrow}
V_Y^k(D_{n+d}^{r_0})
\stackrel{\varphi}{\longrightarrow}
V_Y^k(M) \rightarrow 0
\]
with $V_Y^k(M) := \varphi(V_Y^k(D_{n+d}))$ 
for any integer $k$.
See \cite{OTalgDmod},\cite{OTresolution} for algorithms. 

\begin{theorem}\label{th:algebraic_inv}
  In the notation above, set $u_i = \varphi(e_i)$, where $e_1,\dots,e_{r_0}$
  are the unit vectors.
  Let $\pp$ be a prime ideal of $K[x]$. 
  Assume that there exists an indicial polynomial
  $b_i(s,x) \in K[x]_\pp[s]$ of $u_i$ along $Y$ at $\pp$ for each $i$.
  Let $k_0 \leq k_1$ be integers (or $k_0 = -\infty$)
  such that the specialization of
  $b_i(j,x)$ ($i=1,\dots,r_0$) to $\kappa(\pp)$ do not vanish
  if an integer $j$   satsifies $j \leq k_0$ or $j > k_1$.
 Define the complex $V^\bullet$ by
 \[
0 \rightarrow
  \frac{V_Y^{k_1}[\mvec_{d+1}](D_{Y\rightarrow X}^{r_{d+1}})}
{V_Y^{k_0}[\mvec_d](D_{Y\rightarrow X}^{r_d})}
\stackrel{\overline\psi_d}{\longrightarrow}
\cdots
\stackrel{\overline\psi_2}{\rightarrow}
\frac{V_Y^{k_1}[\mvec_1](D_{Y\rightarrow X}^{r_1})}
{V_Y^{k_0}[\mvec_1](D_{Y\rightarrow X}^{r_1})}
\stackrel{\overline\psi_1}{\rightarrow}
\frac{V_Y^{k_1}(D_{Y\rightarrow X}^{r_0})}
{V_Y^{k_0}(D_{Y\rightarrow X}^{r_0})}
\rightarrow 0  .
\]
Then $K[x]_\pp\otimes_{K[x]}(\tau^{\geq -d}V^\bullet)$ is isomorphic to
$K[x]_\pp\otimes_{K[x]}\mathbb{L}\iota^*(M)$ 
in $D^{-}(\mathrm{Mod}(K[x]_\pp\otimes_{K[x]}D_n))$.
  In particular, if the specialization of $b_i(j,x)$ to $\kappa(\pp)$
  does not vanish for any $j \in \Z$ and
  any $i=1,\dots,r_0$, then $K[x]_\pp\otimes_{K[x]}\mathbb{L}\iota^*(\Msc) = 0$ holds
in $D^{-}(\mathrm{Mod}(K[x]_\pp\otimes_{K[x]}D_n))$.
\end{theorem}

The proof is similar to that of Theorem \ref{th:inv}.
If $K$ is a subfield of $\C$ and $\pp = Z_\C(p)$ with some
$p\in \C^n$, then we have
\[
\mathbb{L}\iota^*(\Dsc_X\otimes_{D_{n+d}}M)_p \simeq 
(\Osc_Y)_p\otimes_{K[x]_\pp}(K[x]_\pp\otimes_{K[x]}\tau^{\geq -d}V^\bullet)
\simeq (\Osc_Y)_p\otimes_{K[x]}\tau^{\geq -d}V^\bullet,
\]
and consequently
\[
H^j(\mathbb{L}\iota^*(D_X\otimes_{D_{n+d}}M)) \simeq (\Osc_Y)_p\otimes_{K[x]}
  H^j(\tau^{\geq -d}V^\bullet)
\]
holds for any $j$.

\subsection{Examples}

Computations of the following examples are done by using a
computer algebra syatem Risa/Asir
(https://www.math.kobe-u.ac.jp/Asir/asir.html).

\begin{example}\label{ex:nonholonomic}\rm
  Let $K$ be a subfield of $\C$. 
  Set $X := K^3 \ni (x,y,t)$, $Y := K^2 \times \{0\} = \{t=0\}$
  and let $\iota : Y \rightarrow X$ be the natural embedding. 
Define a left $D_3$-module $M$ by
  \[
M := D_3u, \quad P_1u = P_2u = 0
  \]
with 
  \[
  P_1 := (x^2-2)(t\partial_t^2 - (x+y)\partial_t),
  \quad P_2 := 2x(\partial_x-\partial_y) + 1.
  \]
  The characteristic variety (see \cite{OakuChar} for an algorithm) of $M$ is
  \begin{multline*}
  \{(x,y,t;\xi dx+\eta dy + \tau dt)\in T^*X \mid x(\xi - \eta) = t\tau = 0\}
\\
= \{ x = t = 0\} \cup \{x = \tau = 0\} \cup \{\xi - \eta = t = 0\} 
\cup \{\xi - \eta = \tau = 0\} ,
  \end{multline*}
  which is of dimension $4$. 
  A Gr\"obner basis of the left ideal $I:= D_3P_1 + D_3P_2$ of $D_3$
  with respect to an ordering adapted to the filtration
  $\{V_Y^j(D_3)\}$ is $G = \{P_2,P_3,P_4\}$ with
\begin{align*}
  P_3 &=  t\partial_t^2 - (x+y)\partial_t,
  \\
  P_4 &= (-4\dx + 4\dy + 3x)t\dt^2 + (4y\dx-4y\dy-3x^2-3yx-2)\partial_t
\end{align*}
and the interesection $\gr_Y^0(I) \cap D_2[s]$ is generated by
(the residue classes of)
\[
s^2 - (x+y+1)s, \quad 2x\dx-2x\dy+1.
\]
    The ideal $J_Y(u)$ is generated by a single element
  $b(s,x,y) := s^2 -(x+y+1)s$.
  Thus $b(s,x,y)$ is the unique indicial polynomial along $Y$
  at every prime ideal of $K[x,y]$.

  Let $\pp$ be a prime ideal of $K[x,y]$ such that
  the specialization $b_\pp(j)$ of $b(j,x,y)$ to $\kappa(p)$
  does not vanish for any integer $j \geq 1$. 
  For example, if $\pp = \langle x-a, y-b\rangle$  with $a,b \in K$,
  then this condition is equivalent to $a +b \neq 0,1,2,\dots$. 
  As another example, set $\pp = \langle x \rangle$.
  Then $b_\pp(j)$ is $j(j-y-1)$ as an element $\kappa(\pp) = K(y)$
  and the condition above is satisfied. 
Now under the assumption above, Theorem \ref{th:algebraic_inv} implies that
$K[x,y]_\pp\otimes_{K[x,y]}H^0(\mathbb{L}\iota^*(M))$ is
isomorphic to the cohomology group of the complex
\[
K[x,y]_\pp\otimes_{K[x,y]}V_Y^0[(0,1,1)](\DYX)
\stackrel{\psi}{\longrightarrow}
K[x,y]_\pp\otimes_{K[x,y]}V_Y^0(\DYX) \longrightarrow 0,
\]
where $\psi$ is the homomorphism induced by $(P_2,P_3,P_4)$.
This complex is equivalent to
\[
K[x,y]_\pp\otimes_{K[x,y]}D_2 \stackrel{\cdot P_2}{\longrightarrow}
K[x,y]_\pp\otimes_{K[x,y]}D_2 \longrightarrow 0. 
\]
Thus we have an isomorphism
\[
K[x,y]_\pp\otimes_{K[x,y]}H^0(\mathbb{L}\iota^*(M)) \simeq
K[x,y]_\pp\otimes_{K[x,y]}(D_2/D_2P_2)
\]
as left $D_2$-modules.
If $K = \C$ and $\pp = Z_\C(p)$ with $p\in X = \C^2$, then we have an
isomorphism
\[
H^0(\mathbb{L}\iota^*(\Dsc_X\otimes_{D_2}M))_p \simeq (\Dsc_X/\Dsc_XP_2)_p. 
\]
\end{example}

\begin{example}\label{ex:BMS}\rm
  Set $X = \Q^3 \ni (x,y,z)$ and $F$ be an ideal of $\Q[x,y,z]$
  generated by
  \[
   f_1 = x^3-y^2z, \quad f_2 = xy(z^2+1). 
  \]
  Let $\iota : X \rightarrow X \times \Q^2$ be the embedding defined by
  \[
  \iota(x,y,z) = (x,y,z,f_1(x,y,z),f_2(x,y,z))
  \]
  and set $Y := X \times \{0\} \subset X \times \Q^2$. 
  Then we have
\[
  \iota_*(\Q[x,y,z]))= D_5/I
\]
with $I$ being the left ideal of $D_5$ generated by
\[
t_1 - f_1,\quad t_2-f_2,\quad
\dx + 3x^2\partial_{t_1} + y(z^2+1)\partial_{t_2},\quad
\dy -2yz\partial_{t_1}+ x(z^2+1)\partial_{t_2},\quad
\dz - y^2\partial_{t_1} + 2xyz\partial_{t_2}.
\]
Let $u$ be the residue class of $1$ in $D_5/I$.
Then a reduced Gr\"obner basis $G$ of the ideal $J_Y(u)$ of $\Q[s,x,y,z]$
with respect to an ordering $\prec$ such that $x,y,z \prec s$
is generated by $31$ elements with degrees $0,1,2,3,4,5,6,7,8,9,10,11,13$
in $s$. Each element of $G$ has the form
$a(x,y,z)b(s)$ with $a\in \Q[x,y,z]$ and $b\in \Q[s]$.
Hence the $b$-function of $u$ along $Y$ at each prime ideal 
of $\Q[x,y,z]$ is also an indicial polynomial
(not that every indicial polynomial is the $b$-function). 
The radical of the ideal $I_k$ in Algorithm \ref{alg:indicial} is as follows:
\begin{align*}
  \sqrt{I_0} &= \langle x,y \rangle \cap \langle x,z\rangle
  \cap \langle f_1,z^2+1\rangle,
\\
  \sqrt{I_k} &= \langle x,y \rangle \quad (1\leq k \leq 5),
\\
\sqrt{I_k} &= \langle x,y,z^2+1 \rangle \quad (6\leq k \leq 12),
\\
\sqrt{I_{13}} &= \langle 1\rangle.
\end{align*}
Thus for a prime ideal $\pp$ of $\Q[x,y,z]$,
the degree of the $b$-function of $u$ along $Y$ at $\pp$ is
$0$ if $\langle x,y \rangle \not\subset \pp$ or
$\langle x,z \rangle \not\subset \pp$ or
$\langle f_1,z^2+1 \rangle \not\subset \pp$;
$1$ if $\langle x,y\rangle \not\subset \pp$ and
($\langle x,z \rangle \subset \pp$ or $\langle f_1,z^2+1 \rangle \subset \pp$);
$6$ if $\langle x,y,z^2+1\rangle \not\subset \pp$ and
$\langle x,y \rangle \subset \pp$;
$13$ if $\langle x,y,z^2+1 \rangle \subset \pp$. 
Set
\begin{align*}
V_0 &= \{(x,y,z) \in \C^3 \mid x=y=0\} \cup \{x=z=0\} \cup \{f_1 = z^2+1 = 0\},
\\
V_1 &= \{(x,y,z) \in \C^3 \mid x=y=0\},
\\
V_6 &= \{(x,y,z) \in \C^3 \mid x=y=z^2+1=0\}. 
\end{align*}
Then
by Theorem \ref{th:BMS}, Proposition \ref{prop:field_ext}, and Theorem \ref{th:comparison}, 
the degree of the Bernstein-Sato polynomial with respect to the coherent ideal
$\Osc_{\C^3}F$ (in the sense of \cite{BMS}) at $p \in \C^3$ is
$0$ if $p \in \C^3 \setminus V_0$, 
$1$ if $p \in V_0 \setminus V_1$,
$6$ if $p \in V_1 \setminus V_6$,
$13$ if $p \in V_6$. 
For this example, 
the $b$-function stays the same on each stratum and 
the Bernstein-Sato polynomial for $\Osc_{\C^3}F$ at each $p$
is given as follows (we set $p = (x,y,z)$):
\begin{align*}
  b_p(s) &= 1 \mbox{ if } p\not\in V_0,
  \\
  b_p(s) &= s+2 \mbox{ if } p \in V_0 \setminus V_1,
\\
  b_p(s) &= (s+1)^2(s+2)(s+3/2)(s+4/3)(s+5/3) \mbox{ if } p \in V_1 \setminus V_6,
  \\
  b_p(s) &= (s+1)^3(s+2)^2(s+3/2)^2(s+4/3)^2(s+5/3)^2(s+7/6)(s+11/6)
  \mbox { if } p \in V_6.
\end{align*}
\end{example}  

\begin{example}\label{ex:AHG}\rm
  Set $X = \C^3$ and
  let $\Msc_A(\beta) = \Dsc_Xu$ be the $A$-hypergeometric system 
for $A = \begin{pmatrix} 1 & 1 & 1 \\ 0 & 1 & 2 \end{pmatrix}$ 
with parameters $\beta = (\beta_1,\beta_2)$; i.e., 
\begin{align*}&
(x_1\partial_{1} + x_2\partial_{2} + x_3\partial_{3} - \beta_1)u 
= (x_2\partial_{2} + 2x_3\partial_{3} - \beta_2)u
= (\partial_{1}\partial_{3} - \partial_{2}^2)u = 0. 
\end{align*}
The singular locus (see \cite{OakuChar} for an algorithm)
of $\Msc_A(\beta)$ is 
\[
\{(x_1,x_2,x_3) \in\C^3 \mid x_1x_3(4x_1x_3-x_2^2) = 0\}. 
\]
Set $f = x_1x_3(4x_1x_3-x_2^2)$ and let
$\iota : X \rightarrow X \times \C$ be the embedding defined by
$\iota(x_1,x_2,x_3) = (x_1,x_2,x_3,f(x_1,x_2,x_3))$. 
Let us compute the $b$-function of $\iota_*(u)$ 
by regarding $\beta_1,\beta_2$ as
variables instead of constants. That is, we regard $M$ as a
left $D_3[\beta_1,\beta_2]$-module. 
A reduced Gr\"obner basis $G$ of $J_X(\iota_*(u))$, which is an ideal of
$\C[x_1,x_2,x_3,\beta_1,\beta_2]$, consists of $22$ elements
with degrees $0,2,3,4,5,6,7,9,10,12$ in $s$.
Contrary to Example \ref{ex:AHG}, some elements of $G$ do not
have the form $a(x,\beta)b(s,\beta)$. 
The radical of the ideal $I_k$ in Algorithm \ref{alg:indicial} is as follows:
\begin{align*}
  \sqrt{I_k} &= \langle x_1 \rangle \cap \langle x_3\rangle
  \cap \langle 4x_1x_3-x_2^2\rangle \quad (k =0,1),
\\
\sqrt{I_2} &= \langle x_1,x_2 \rangle \cap \langle x_1, x_3\rangle
  \cap \langle x_2, x_3 \rangle, 
  \\
\sqrt{I_k} &= \langle x_1,x_2 \rangle \cap \langle x_2, x_3\rangle
\quad (3 \leq k \leq 6),
\\
\sqrt{I_k} &= \langle x_1,x_2,x_3 \rangle \quad (7\leq k \leq 11),
\\
\sqrt{I_{12}} &= \langle 1\rangle.
\end{align*}
Since the generators of these ideals do not involve $\beta_1,\beta_2$,
the stratification of $X \times \C^2$ with respect to the degree of
an indicial polynomial is the direct product of that in $X$ and $\C^2$.  
An indicial polynomial $b_p(s)$ of $\iota_*(u)$ as an element of
$D_3[\beta_1,\beta_2]$-module $\iota_*(M)$ along $X \simeq X \times \{0\}$
at $(p,b) = (p_1,p_2,p_3,b_1,b_2) \in X \times \C^2$,
which is the point
corresponding to a maximal ideal of $\C[x,y,z,\beta_1,\beta_2]$
is as follows: 
\begin{align*}
  b_p(s) &= 1 \mbox{ if } f(p) \neq 0,
  \\
  b_p(s) &= s(s-\beta_1-1/2) \mbox{ if } f(p) = 0 \mbox{ and }
  p_1p_3 \neq 0, 
  \\
  b_p(s) &= s(s-\beta_1+\beta_2)(s+\beta_1-\beta_2)
  \mbox{ if } p_2 \neq 0 \mbox{ and } p_1 = p_3 = 0, 
  \\
  b_p(s) &= s(s+\beta_1-\beta_2)(s-\beta_1-1/2)(s-\beta_2/4)
  (s-\beta_2/4-1/2)(s-\beta_2/4-1/4)
 \\&\quad
  (s-\beta_2/4+1/4)
   \mbox{ if } p_1\neq 0 \mbox{ and } p_2 = p_3=0,
  \\
  b_p(s) &= s(s-\beta_1+\beta_2)(s-\beta_1-1/2)(s-\beta_1/2 +\beta_2/4)
  (s-\beta_1/2 +\beta_2/4-1/2)
\\&\quad
  (s-\beta_1/2+\beta_2/4-1/4)(s-\beta_1/2+\beta_2/4+1/4) 
\mbox{ if } p_3\neq 0 \mbox{ and } p_1=p_2=0,
\\
b_p(s) &=
s(s-\beta_1+\beta_2)(s+\beta_1-\beta_2)(s-\beta_1-1/2)(s-\beta_2/4)
\\&\quad
(s-\beta_2/4-1/2)(s-\beta_2/4-1/4)(s-\beta_2/4+1/4)(s-\beta_1/2+\beta_2/4)
\\&\quad
(s-\beta_1/2+\beta_2/4-1/2)(s-\beta_1/2+\beta_2/4-1/4)
(s-\beta_1/2+\beta_2/4+1/4) 
\\&\quad
\mbox{ if } p = (0,0,0). 
  \end{align*}
The $b$-function of $u$ with respect to the ideal $\Osc_Xf$
in the sense of Theorem \ref{th:BMS} at $p$ is
given by $b_p(s-1)$.

On the other hand, the indicial polynomial of $u$ along $x_1= 0$
at $(p,b)$ with $p_1=0$ is
$s(s-\beta_1+\beta_2)$, the indicial polynomial of $u$ along $x_3=0$
at $(p,b)$ with $p_3=0$ is $s(s+\beta_1-\beta_2)$, and the indicial polynomial of $u$ along $x_2=0$ at $(p,b)$ with $p_2=0$ is $s(s-1)$.
In order to compute the indicial polynomial along $g:= 4x_1x_3-x_2^2 = 0$,
now let $\iota : X \rightarrow X \times \C^2$ be the embedding
defined by $g$ instead of $f$. Then the indicial polynomial of
$\iota_*(u)$ along $X \times \{0\}$ at $(p,b)$ is $s(s-\beta_1-1/2)$
if $g(p) = 0$. Hence by Theorem \ref{th:embedding},
the indicial polynomial of $u$ along the hypersurface $g=0$
at $(p,b)$ with its non-singular point $p$ (i.e., $g(p) = 0$ and $p \neq 0$)
is $s(s-\beta_1-1/2)$. 
\end{example}


\end{document}